\documentclass[preprint,12pt]{elsarticle}

\usepackage{graphicx}
\usepackage{amsthm}
\usepackage{amsmath}
\usepackage{amssymb}
\usepackage{amsfonts}
\usepackage{mathtools}
\usepackage{color}
\usepackage{url}
\usepackage{hyperref}
\usepackage{algorithm}
\usepackage{algorithmic}
\usepackage{mathdots}

\newtheorem{theorem}{Theorem}[section]
\newtheorem{remark}[theorem]{Remark}
\newtheorem{proposition}[theorem]{Proposition}
\numberwithin{equation}{section}

\newcommand{\norm}[1]{\left\lVert#1\right\rVert}
\newcommand{\bm}[1]{\mathbf{#1}}

\newcommand{\R}{{\mathbb{R}}}

\newcommand{\cV}{{\mathcal{V}}}

\DeclareMathOperator{\diag}{diag}
\DeclareMathOperator{\tr}{trace}
\DeclareMathOperator{\ran}{rank}

\journal{Journal of Computational and Applied Mathematics}

\begin{document}           

\begin{frontmatter}

%\pagestyle{myheadings}
%\markboth{A. Concas, S. Noschese, L. Reichel, and G. Rodriguez}{}

\title{A spectral method for bipartizing a network and detecting a large
anti-community}
%\footnotetext{Version \today}

\author[address1]{A. Concas\corref{cor1} \fnref{CR,CNR,Concas,GNCS}}
\ead{anna.concas@unica.it}
\cortext[cor1]{Corresponding author}

\author[address2]{S. Noschese\fnref{CNR,GNCS}}
\ead{noschese@mat.uniroma1.it}

\author[address3]{L. Reichel\fnref{Reichel}}
\ead{reichel@math.kent.edu}

\author[address1]{G. Rodriguez\fnref{CR,CNR,GNCS}}
\ead{rodriguez@unica.it}

\address[address1]{Department of Mathematics and Computer Science, University
of Cagliari, \\ Viale Merello, 92, 09123 Cagliari, Italy}
\address[address2]{Department of Mathematics ``Guido Castelnuovo'',
Sapienza University of Rome, \\ P.le A. Moro, 2, I-00185 Roma, Italy}
\address[address3]{Department of Mathematical Sciences, Kent State University,
Kent, OH 44242, USA}

\fntext[CR]{Partially supported by the Fondazione di Sardegna 2017 research
project ``Algorithms for Approximation with Applications (Acube)''
and the Regione Autonoma della Sardegna research project ``Algorithms and
Models for Imaging Science (AMIS)'' (intervento finanziato con risorse FSC
2014-2020 - Patto per lo Sviluppo della Regione Sardegna).}
\fntext[CNR]{Partially supported  by the INdAM-GNCS
research project ``Metodi numerici per problemi mal posti''.}
\fntext[Reichel]{Partially supported by NSF grants DMS-1729509 and DMS-1720259.}
\fntext[Concas]{Anna Concas gratefully acknowledges Sardinia Regional Government for the financial support of her Ph.D. scholarship (P.O.R. Sardegna F.S.E. Operational Programme of the Autonomous Region of Sardinia, European Social Fund 2014-2020 - Axis III Education and Formation,
Objective 10.5, Line of Activity 10.5.12).}
\fntext[GNCS]{Member of the INdAM Research group GNCS.}

\begin{abstract} 
Relations between discrete quantities such as people, genes, or streets can be
described by networks, which consist of nodes that are connected by edges.
Network analysis aims to identify important nodes in a network and to uncover
structural properties of a network. A network is said to be bipartite if its
nodes can be subdivided into two nonempty sets such that there are no edges
between nodes in the same set. It is a difficult task to determine the closest
bipartite network to a given network. This paper describes how a given network
can be approximated by a bipartite one by solving a sequence of fairly simple
optimization problems.
The algorithm also produces a node permutation which makes the possible bipartite nature
of the initial adjacency matrix evident, and identifies the two sets of nodes.
We finally show how the same procedure can be used to detect the presence of a
large anti-community in a network and to identify it.
\end{abstract}


\end{frontmatter}
%%%%%%%%%%%%%%%%%%%%%%%%%%%%%%%%%%%%%%%%%%%%%%%%%%%%%%%%%%%%%%%%%%%%%

\section{Introduction}\label{sec:intro}

Networks describe how discrete quantities such as genes, people, proteins, or streets are
related. They arise in many applications, including genetics, epidemiology, energy 
distribution, and telecommunication; see, e.g., \cite{Es,Ne} for discussions on networks 
and their applications. Networks are represented by graphs 
$\mathcal{G}=\{\mathcal{V},\mathcal{E},\mathcal{W}\}$, which are determined by a set 
of vertices (nodes) $\mathcal{V}=\{v_i\}_{i=1}^n$, a set of edges 
$\mathcal{E}=\{e_k\}_{k=1}^m$, and a set of positive weights 
$\mathcal{W}=\{w_k\}_{k=1}^m$. Here $e_k=(i_k,j_k)$ represents an edge from vertex 
$v_{i_k}$ to vertex $v_{j_k}$. The weight $w_k$ is associated with the edge $e_k$; a large
value of $w_k>0$ indicates that edge $e_k$ is important. For instance, in a road network, 
the weight $w_k$ may be proportional to the amount of traffic on the road that is 
represented by the edge $e_k$. In this paper, we 
consider connected undirected graphs without self-loops and multiple edges. In particular,
all edges represent ``two-way streets,'' i.e., if $(i_k,j_k)$ is an edge, then so is 
$(j_k,i_k)$. The weights associated with these edges are assumed to be the same. In
unweighted graphs all weights are set to one.

We will represent a graph $\mathcal{G}$ with $n$ nodes by its adjacency matrix 
$A=[a_{i,j}]_{i,j=1}^n$, where
\[
a_{i,j}=\begin{cases} w_k, & \text{if there is an edge $e_k$ between the nodes $v_i$ and 
$v_j$ with weight $w_k$},\\
0, & \text{otherwise}. \end{cases}
\]
Since $\mathcal{G}$ is undirected and the weights associated with each direction of 
an edge are the same, the matrix $A$ is symmetric. The largest possible number of edges of
an undirected graph with $n$ nodes without self-loops is $n^2-n$, but typically the 
actual number of edges, $m$, of such graphs that arise in applications is much smaller. 
The adjacency matrix $A$, therefore, is generally very sparse.

A graph $\mathcal{G}$ is said to be \emph{bipartite} if the set of vertices 
$\mathcal{V}$ that make up the graph can be partitioned into two disjoint nonempty 
subsets $\mathcal{V}_1$ and $\mathcal{V}_2$ (with 
$\mathcal{V}=\mathcal{V}_1\cup\mathcal{V}_2$), such that any edge starting at a vertex
in $\mathcal{V}_1$ points to a vertex in $\mathcal{V}_2$, and vice versa. 
This, in particular, excludes the presence of self-loops in a bipartite graph.

Bipartivity is
an important structural property. It has been studied also as the $2$-coloring
problem \cite{BM}. In fact determining if a graph can be colored with 2 colors
is equivalent to determining whether or not the graph is bipartite, and thus
testing if a network is bipartite or not is computable in linear time using
breadth-first or depth-first search algorithms.
It is therefore interesting to determine a bipartite approximation of a
non-bipartite graph, or measure the distance of a non-bipartite graph from
being bipartite.
We say that a splitting of the set of vertices 
$\mathcal{V}$ of a weighted undirected graph $\mathcal{G}$ into two disjoint nonempty 
subsets $\mathcal{V}_1$ and $\mathcal{V}_2$ (with 
$\mathcal{V}=\mathcal{V}_1\cup\mathcal{V}_2$), is a best bipartization of 
$\mathcal{G}$ if the sum of the weights $w_k$ associated with edges $e_k=(i,j)$ that
point from vertices $v_i$ in $\mathcal{V}_\ell$ ($\ell=1,2$)
to vertices $v_j$ in the same set $\mathcal{V}_\ell$ is 
minimal.
Such edges $e_k$ are called ``frustrated'', and computing the minimum
number of edges whose deletion makes the graph bipartite is an NP-hard
optimization problem \cite{Yann}.
We remark that the above definition is analogous to the definition of a
best bipartization of an undirected unweighted graph proposed by Estrada and 
G\'omez--Garde\~nes \cite{EGG}, where the \emph{spectral bipartivity index} of
a network with adjacency matrix $A$ is defined as
\begin{equation}\label{bipind}
b_s = \frac{\tr(\exp(-A))}{\tr(\exp(A))}.
\end{equation}
This measure also can be applied to the weighted graphs considered in the present paper.

The problem of discovering approximately bipartite structures in graphs and
networks has been considered by various authors. Most popular approaches are
based on the eigendecomposition of the Laplacian and signless Laplacian
matrices. Other spectral approaches consider the adjacency matrix associated to
the graph. In the case of a symmetric bipartite  adjacency matrix, the
signs of the entries of an eigenvector associated with the smallest
eigenvalue can be used to partition the graph, i.e., nodes that correspond to positive 
entries belong to one set, and nodes that correspond to negative entries belong to the other 
set; see \cite{R}. In case the smallest eigenvalue is multiple, the splitting of the 
nodes may vary according to the considered vector in the associated eigenspace.
In \cite{MBHG} the presence of $\pm$ pairs in the spectrum of the adjacency
matrix of a bipartite graph is exploited in order to identify approximated
bipartite structures within protein-protein interaction undirected networks;
see also \cite{TVH} for a spectral approach that can be used to discover approximately bipartite  substructures in directed graphs.
%A similar problem arises in graph partitioning based on the Fiedler vector; see
%\cite{CFR} for a recent discussion. 

We are interested in developing a numerical 
method for determining a ``good'' bipartization $(\cV_1,\cV_2)$, i.e., a
bipartization for which the sum 
of the weights $w_k$ associated with the edges $e_k=(i,j)$ that point from a vertex 
$v_i$ in $\mathcal{V}_1$ to a vertex $v_j$ in $\mathcal{V}_2$, or vice versa, is fairly small.
The algorithm is approximated, or ``heuristic'', in the sense that it does not
necessarily produce the best possible bipartization. 

As it will be made clear in the following, the same bipartization method may be
used for the identification of large anti-communities.
A community is a group of nodes which are highly connected among themselves,
but are less connected to the rest of the network, or isolated from it.
Conversely, an anti-community is a node set that is loosely connected internally,
but has many external connections \cite{Esbook}; see \cite{FasTud}, where
a spectral method is used to detect communities and anti-communities.
Community and anti-community detection in networks is an important problem with  
applications in various fields, including physics, computer science, and
social sciences \cite{CYC,LLDM,PDFV,RAK,YCL}. Although the identification of
communities is predominant in the investigation of meso-scale structures in
networks, the detection of the so-called core-periphery structures, whose most
popular notion was developed by Borgatti and Everett \cite{BE}, attracts a
continuing interest also in the mathematical community; see also \cite{RPFM}.
For our purposes, the identification of a single large anti-community can be
understood as that of a core-periphery structure in the given network.

This paper is organized as follows. Section \ref{sec2} discusses some
properties of bipartite graphs and Section \ref{sec2.5} describes an algorithm
for determining a ``good'' bipartization.
An application of the bipartization method to the identification of large
anti-communities is discussed in Section \ref{sec3}.
%Symmetric tridiagonal matrices with nonnegative off-diagonal entries and
%vanishing diagonal entries are adjacency matrices for particular undirected
%weighted bipartite graphs. Section \ref{sec4} discusses some properties of
%these adjacency matrices and associated graphs.
Finally, Section \ref{sec5} presents computed examples and two case studies,
while Section \ref{sec6} contains concluding remarks.

\section{Approximating the spectral structure of a bipartite graph}\label{sec2}

This section discusses some properties of the adjacency matrix for an undirected bipartite 
graph. Some inequalities that are useful for the design of our
bipartization method also will be introduced. The discussion in the first part
of the section
assumes that the vertices are suitably ordered. Subsequently, we will describe how to
achieve such an ordering.
%An algorithm for our bipartization method concludes the section.

Assume for the moment that the undirected graph 
$\mathcal{G}=\{\mathcal{V},\mathcal{E},\mathcal{W}\}$ is bipartite, i.e., its vertex 
set $\mathcal{V}$ can be split into two disjoint nonempty subsets $\mathcal{V}_1$ and 
$\mathcal{V}_2$ with $n_1$ and $n_2$ nodes, respectively, such that there are no edges 
between the nodes in $\mathcal{V}_1$ and between the nodes in $\mathcal{V}_2$. We may 
assume that $n_1\geq n_2$, otherwise we interchange the sets $\mathcal{V}_1$ and 
$\mathcal{V}_2$.

Let the vertices in the set $\mathcal{V}$ be ordered so that the first $n_1$ of them
belong to the set $\mathcal{V}_1$ and the remaining $n_2$ vertices belong to 
$\mathcal{V}_2$. Then the adjacency matrix for the graph $\mathcal{G}$ is of the form
\begin{equation}\label{adjbip}
A_B = \begin{bmatrix}
O_{n_1} & C \\ C^T & O_{n_2}
\end{bmatrix},
\end{equation}
where $O_k$ denotes the  $k\times k$ zero matrix, and 
$C=[c_{i,j}]\in{\mathbb R}^{n_1\times n_2}$ with $c_{i,j}>0$ if the
node $v_i$ in $\mathcal{V}_1$ is connected to the node $v_{n_1+j}$ in
$\mathcal{V}_2$; otherwise $c_{i,j}=0$. 
%This can always be achieved by symmetric row and column permutation. 

We adapt to our notation a known result in graph theory; see, 
e.g.,~\cite[Theorem~3.14]{bapat}.
\begin{proposition}\label{theor0}
Let $\mathcal{G}$ be an unweighted graph with $n$ nodes. Then $\mathcal{G}$ is bipartite and 
the adjacency matrix can be partitioned as in \eqref{adjbip} if and only if the 
spectrum of the adjacency matrix is symmetric with respect to the origin, i.e., 
\begin{equation}\label{specab}
\sigma(A_B)=\{\lambda_1,\ldots,\lambda_{n_2},
\underbrace{0,\ldots,0}_{n_1-n_2},-\lambda_{n_2},
\ldots,-\lambda_1\},
\end{equation}
for some integers $n_1\geq n_2$ and non-negative numbers $\lambda_1\geq\lambda_2\geq\cdots\geq\lambda_{n_2}$. 
The claim holds true also for weighted graphs, as long as the weights are positive.
\end{proposition}

\begin{proof}
For the sake of clarity, we give a quick sketch of the proof.
The necessary condition is straightforward.
The sufficient condition can be proved by noting that, for $k=0,1,\ldots$,
$\tr(A_B^{2k+1})=0$ if the spectrum is symmetric.
Then, the positivity of the weights implies that $(A_B^{2k+1})_{i,i}=0$, that is,
the graph is bipartite since it does not contain odd cycles.
\end{proof}

%When the graph $\mathcal{G}$ is unweighted, the converse of the above result is
%valid too; see, e.g.,~\cite[Theorem~3.14]{bapat}.
%This does not happen if the graph is weighted.
%Indeed, if $D$ is the diagonal matrix containing the eigenvalues in
%\eqref{specab} and $Q$ is a random orthogonal matrix of order $n_1+n_2$, the
%matrix $QDQ^T$ generally does not have the structure \eqref{adjbip}.
\begin{remark}\rm
Under the assumption of Proposition \ref{theor0}, it is immediate to verify
that if $\lambda$ is a nonzero eigenvalue of $A_B$ and
$\bm{q}=\left[\begin{smallmatrix}\bm{x}\\ \bm{y}\end{smallmatrix}\right]$, with 
$\bm{x}\in\R^{n_1}$ and $\bm{y}\in\R^{n_2}$, is an associated eigenvector,
then $\left(-\lambda,\left[\begin{smallmatrix}\bm{x}\\ 
-\bm{y}\end{smallmatrix}\right]\right)$ is an eigenpair, too.
This implies that $\lambda$ is a singular value of the block $C$ in
\eqref{adjbip}, while $\bm{x}$ and $\bm{y}$ are its left and right singular
vectors, respectively, if scaled to be of unit length.
\end{remark}

Let $n=n_1+n_2$ with $n_1\geq n_2\geq 1$. Then, the above observation gives us
the possibility to describe the spectral structure of $A_B$ in terms of the singular value decomposition of $C$; see also
\cite[Section~8.6.1]{GVL}.
Let $C=X\widetilde{D}Y^T$ be a singular value decomposition of $C$,
where $\widetilde{D}\in\R^{n_1\times n_2}$ has
$D=\diag(\lambda_1,\dots,\lambda_{n_2})$ as its upper block, and
$X=[X_1,X_2]\in\R^{n_1\times n_1}$ and $Y\in\R^{n_2\times n_2}$ are orthogonal matrices
with $X_1\in\R^{n_1\times n_2}$. Introduce the diagonal matrix
\[
\mathcal{D} = \diag(D,O_{n_1-n_2},-D),
\]
and the orthogonal matrix
\begin{equation}\label{Qmat}
Q = \begin{bmatrix} U_1 & U_2 & U_1 \\ V & O_{n_2,n_1-n_2} & -V \end{bmatrix},
\end{equation}
where $U_1=\frac{1}{\sqrt{2}}X_1$, $U_2=X_2$, and $V=\frac{1}{\sqrt{2}}Y$,
with $U_1^TU_1=V^TV=\frac{1}{2}I_{n_2}$ and $U_2^TU_2=\frac{1}{2}I_{n_1-n_2}$.
Then, the spectral factorization 
\begin{equation}\label{spefac}
A_B = Q\mathcal{D} Q^T,
\end{equation}
takes the form 
\begin{equation}\label{spectrg}
\begin{bmatrix} U_1 & U_2 & U_1 \\ V & O_{n_2,n_1-n_2} & -V \end{bmatrix}
\diag(D, O_{n_1-n_2} , -D)
\begin{bmatrix} U_1 & U_2 & U_1 \\ V & O_{n_2,n_1-n_2} & -V \end{bmatrix}^T.
\end{equation}

In the special case when $n_1=n_2$, the submatrices of \eqref{Qmat} with
$n_1-n_2$ columns disappear, and the spectral factorization \eqref{spectrg}
simplifies to
\[
A_B = \begin{bmatrix}
U_1 & U_1\\
V & -V
\end{bmatrix}\begin{bmatrix}
D & 0\\
0 & -D
\end{bmatrix}\begin{bmatrix}
U_1 & U_1\\
V & -V
\end{bmatrix}^T,
\]
with $U_1U_1^T = VV^T = \frac{1}{2}I_{n_1}$.

Now, let $A$ be an adjacency matrix of an undirected graph. We would like to approximate the 
graph by a bipartite one, and therefore seek to approximate $A$ by a matrix of the form 
$A_B$. We do this in several steps and first show some inequalities that are applicable
to diagonal eigenvalue matrices.

\smallskip
\begin{proposition}\label{theor1g}
Let $\alpha_1\geq\alpha_2\geq\dots\geq\alpha_\ell$ be a nonincreasing real sequence and 
let $\beta_1,\beta_2,\dots,\beta_\ell$ be another real sequence. The distance between these
sequences measured in the least squares sense,
\begin{equation}\label{distg}
\biggl(\sum_{i=1}^\ell (\alpha_i-\beta_i)^2\biggr)^{1/2},
\end{equation}
is minimal if and only if the $\beta_i$ are in nonincreasing order, i.e., if 
$\beta_1\geq\beta_2\geq\dots\geq\beta_\ell$.
\end{proposition}

\smallskip
\begin{proof}
Assume that both sequences are in nonincreasing order and that the distance can be
reduced by changing the order of the $\beta_i$. Consider the pairs 
$(\alpha_1,\beta_1)$ and $(\alpha_2,\beta_2)$. Then
$$
(\alpha_1-\beta_2)^2+(\alpha_2-\beta_1)^2\leq (\alpha_1-\beta_1)^2+(\alpha_2-\beta_2)^2 
$$
is equivalent to 
$$
\alpha_2(\beta_1-\beta_2)\geq \alpha_1(\beta_1-\beta_2).
$$
Assume $\beta_1>\beta_2$. Then $\alpha_2\geq\alpha_1$, which is a
contradiction unless $\alpha_1=\alpha_2$. If the $\beta_j$ are ordered
arbitrarily, then we can reorder these coefficients pairwise until they form a
nonincreasing sequence. Each pairwise swap reduces \eqref{distg}.
\end{proof}

In our application of Proposition \ref{theor1g}, we let 
$\alpha_1\geq\alpha_2\geq\dots\geq\alpha_n$ be the eigenvalues of the adjacency matrix 
$A\in{\mathbb R}^{n\times n}$. The graph associated with this matrix might not be 
bipartite. We would like the sequence of eigenvalues of the matrix 
$A_B\in{\mathbb R}^{n\times n}$, given by \eqref{adjbip}, to be close to the sequence 
$\alpha_1,\alpha_2,\dots,\alpha_n$ and appear in $\pm$ pairs. By Proposition
\ref{theor1g},
we know that the eigenvalues $\beta_1,\beta_2,\dots,\beta_n$ of $A_B$ should be in
nonincreasing order, and by Proposition~\ref{theor0} they vanish or appear in $\pm$ pairs.
We know from \eqref{spectrg} that at least $n_1-n_2$ eigenvalues of $A_B$ should be zero.

\smallskip
\begin{proposition}\label{theor2g}
Let $\{\alpha_j\}_{j=1}^n$, with $n=n_1+n_2$ and $n_1\geq n_2$, be a real nonincreasing 
sequence. Then the sequence $\{\beta_j\}_{j=1}^n$ with elements
\begin{equation}\label{betag}
\beta_j  = \begin{cases}
\frac{1}{2}(\alpha_j-\alpha_{n-j+1}), & j = 1,2,\dots,n_2, \\
0, & j = n_2+1,\dots,n_1, \\
-\beta_{n-j+1}, & j = n_1+1,\dots,n, \\
\end{cases}
\end{equation}
is the closest sequence to $\{\alpha_j\}_{j=1}^n$ in the least squares sense consisting of
at least $n_1-n_2$ zeros and nonvanishing entries appearing in $\pm$ pairs.
\end{proposition}

\smallskip
\begin{proof}
The sequence $\{\beta_j\}_{j=1}^{n}$ consists of $n_1-n_2$ zero values and $n_2$ $\pm$ pairs. Indeed, we have
$$
\beta_j-\beta_{j+1} =
\begin{cases}
\frac{1}{2}(\alpha_j-\alpha_{j+1})
	+\frac{1}{2}(\alpha_{n-j}-\alpha_{n-j+1}), 
	\quad & 1\leq j\leq n_2-1, \\
\frac{1}{2}(\alpha_{n_2} - \alpha_{n_1+1}), & j=n_2, \\
0, & n_2+1\leq j\leq n_1-1, \\
\beta_{n_2}, & j=n_1, \\
\beta_{n-j} - \beta_{n-j+1},
	\quad & n_1+1\leq j\leq n-1,
\end{cases}
$$
and it follows that the sequence is nonincreasing. It remains to establish that the 
$\beta_j$ defined by \eqref{betag} are the best possible. Consider the minimization 
problems
\begin{equation}
\begin{cases}\label{ming}
\min_{\beta} \Bigl((\alpha_j-\beta)^2+(\alpha_{n-j+1} +\beta)^2\Bigr), \quad 
&1\leq j \leq n_2,\\
\min_{\beta}\,(\beta^2), \quad 
&n_2+1\leq j \leq n_1.
\end{cases}
\end{equation}
The solution sequence
$\{\beta_j\}_{j=1}^n$ is given by \eqref{betag}. Thus, the $\beta_j$ form a nonincreasing sequence 
consisting of $n_1-n_2$ zero values and $n_2$ $\pm$ pairs. It is the closest such sequence to the sequence 
$\{\alpha_j\}_{j=1}^{n}$ in the sense that it solves the minimization problems 
\eqref{ming}.
\end{proof}
\smallskip

We would like to determine an approximation of the matrix $A$ by a matrix of the form 
\eqref{adjbip}, where we allow row and column permutations of the latter
matrix. Define the spectral factorization 
\[
A_B=W_B\Lambda_B W_B^T,\quad 
\Lambda_B=\diag(\lambda_1^{(B)},\lambda_2^{(B)},\ldots,\lambda_n^{(B)}),
\]
where $W_B$ is an orthogonal matrix and the eigenvalues are ordered according to
$$
\lambda_1^{(B)} \geq \lambda_2^{(B)} \geq \cdots \geq \lambda_n^{(B)}.
$$
We remark that only the first $n_1$ eigenvalues are ordered as in \eqref{spefac}.

Let us initially assume that the nonzero eigenvalues are distinct.
If the eigenvectors are made unique, e.g., by making their first component
positive, a comparison with~\eqref{spectrg} shows that
\begin{equation}\label{eigstructg}
W_B = \begin{bmatrix}
U_1 & U_2 & U_1Z \\
V & O & -VZ
\end{bmatrix}, \qquad
\Lambda_B = \begin{bmatrix}
D & O & O \\
O & O_{n_1-n_2} & O \\
O & O & -ZDZ
\end{bmatrix},
\end{equation}
where $Z$ is the \emph{flip} matrix
$$
Z = \begin{bmatrix} O & & 1 \\ & \iddots \\ 1 & & O
\end{bmatrix}\in\R^{n_2\times n_2}.
$$

In the presence of multiple nonzero eigenvalues, the corresponding
eigenvectors are not uniquely determined, so the spectral factorization
\eqref{eigstructg} is only one of several possible distinct factorizations.

Let 
\begin{equation}\label{Aspecg}
A=W\Lambda W^T,\quad \Lambda=\diag(\lambda_1,\lambda_2,\ldots,\lambda_n),
\end{equation}
be a spectral factorization of $A$ with an orthogonal eigenvector matrix $W$ and the 
eigenvalues ordered according to 
\begin{equation}\label{Alambdag}
\lambda_1\geq \lambda_2\geq\cdots\geq\lambda_n.
\end{equation}
Partition the eigenvector matrix $W$ conformally with the eigenvector matrix $W_B$ of 
$A_B$, i.e., 
$$
W = \begin{bmatrix}
W_{11} & W_{12} & W_{13}\\
W_{21} & W_{22} & W_{23}
\end{bmatrix}.
$$

We would like to to approximate the eigenvector matrix $W$ of $A$ by the eigenvector 
matrix $W_B$ of $A_B$. This suggests that we solve the minimization problem
\begin{equation}\label{minW}
\min_{\substack{U_1^TU_1=V^TV=\frac{1}{2}I_{n_2}\\
U_2^TU_2=\frac{1}{2}I_{n_1-n_2}}} \norm{\begin{bmatrix}
U_1 & U_2 & U_1 \\
V & O & -V
\end{bmatrix}-\begin{bmatrix}
W_{11} & W_{12} & W_{13}Z \\
W_{21} & W_{22} & W_{23}Z
\end{bmatrix}}_F,
\end{equation}
where $\norm{\cdot}_F$ denotes the Frobenius norm. This problem splits into the
three independent problems
\begin{align}
\label{prb1g}
\min_{U_1^TU_1=\frac{1}{2}I_{n_2}} & \{\norm{U_1-W_{11}}_F^2
	+\norm{U_1-W_{13}Z}_F^2\}, \\
\label{prb2g}
\min_{V^TV=\frac{1}{2}I_{n_2}} & \{\norm{V-W_{21}}_F^2+\norm{V+W_{23}Z}_F^2\}, \\
\label{prb3g}
\min_{U_2^TU_2=\frac{1}{2}I_{n_1-n_2}} & \{\norm{U_2-W_{12}}_F^2\}.
\end{align}

Problem \eqref{prb1g} can be written as 
\begin{equation}\label{min2g}
\min_{X_1^TX_1=I_{n_2}}\left\{\norm{X_1-\sqrt{2}W_{11}}_F^2
	+\norm{X_1-\sqrt{2}W_{13}Z}_F^2\right\}.
\end{equation}
The following result shows how we can easily solve this problem.

\begin{proposition}\label{prp1g}
The solution of problem \eqref{min2g} can be determined by computing the singular value
decomposition of $W_{11}+W_{13}Z$ and setting all singular values to one.  
\end{proposition}

\begin{proof}
Consider the problem
$$
\min_{X^TX=I} \|X-W\|_F^2.
$$
It can be written as
$$
\min_{X^TX=I} \{\tr(X^TX) - 2 \tr(X^TW) + \tr(W^TW)\}.
$$
The first and last terms are independent of $X$. Therefore we obtain the equivalent 
linear minimization problem
$$
\min_{X^TX=I} \{-\tr(X^TW)\}.
$$

Similarly, the linear problem associated to the minimization problem \eqref{min2g} is 
given by
\begin{equation}\label{minprobg}
\min_{X_1^TX_1=I_{n_2}} \{-\tr(X_1^T(W_{11}+W_{13}Z))\}.
\end{equation}
Hence, the problem \eqref{min2g} is equivalent to determining the closest orthogonal matrix
in the Frobenius norm to the matrix $W_{11}+W_{13}Z$. The solution is given by
computing the singular value decomposition $P\Sigma Q^T$ of $W_{11}+W_{13}Z$
and setting $X_1=PQ^T$; see 
\cite[Theorem 4.1]{H89} for a proof of the latter statement.
\end{proof}
\smallskip

The minimization problems \eqref{prb2g} and \eqref{prb3g} are solved similarly.
This gives the eigenvector matrix in the spectral factorization \eqref{spectrg}.

\begin{remark}\rm
We note that if $P\Sigma Q^T$ denotes the singular value decomposition of
$W_{11}+W_{13}Z$, then we can express its polar decomposition by
\[
W_{11}+W_{13}Z = (PQ^T)(Q\Sigma Q^T).
\]
Since the first factor $PQ^T$ is the minimizer of \eqref{minprobg}, the
deviation of $Q\Sigma Q^T$ from the identity matrix measures the quality of the
approximation.
\end{remark}

\begin{remark}\rm
If some of the nonzero eigenvalues of $A$ in \eqref{Aspecg} are
multiple, the corresponding columns of $W_{11}$, $W_{21}$, $W_{13}$,
and $W_{23}$, are not uniquely determined. Anyway, when approximating
$W_{11}+W_{13}Z$ by $X_1$, and $W_{21}-W_{23}Z$ by $Y$, those columns
contain linear combinations of the previous ones, and so they belong to the
same space. Then, the approximations $X_1$ and $Y$ will make factorization
\eqref{eigstructg} valid.
\end{remark}

\section{A spectral bipartization method}\label{sec2.5}

We give here an outline of a spectral bipartization method, 
based on the above results.
It exploits the spectral structure \eqref{spectrg} of a bipartite graph to
determine a node permutation that separates the two sets $\cV_1$ and $\cV_2$,
and to construct a bipartite approximation to a connected undirected graph
$\mathcal{G}$, having a perturbed bipartite structure.
The algorithm is exact whenever the input is the adjacency matrix of a
bipartite graph, however
it has to be considered ``heuristic'', as we were not able to prove
a complete convergence result for it, apart from the spectrum approximation
theorems in Section~\ref{sec2}.

There are three problems at hand:
estimating the cardinality of the sets $\mathcal{V}_1$ and $\mathcal{V}_2$, 
suitably ordering the nodes in $\mathcal{G}$, and, finally, 
approximating the adjacency matrix by a matrix of the form (\ref{adjbip}).
Let $A$ be the adjacency matrix of $\mathcal{G}$, and assume the spectral
factorization
\[
A=W\mathcal{D} W^T,\quad 
\mathcal{D}=\diag(\lambda_1,\lambda_2,\ldots,\lambda_n),
\]
is available, where $W$ is an orthogonal matrix and the eigenvalues are
ordered by increasing absolute value.

\begin{enumerate}
\item
The first step of our algorithm consists of finding the cardinality $n_1$ and
$n_2$ of the two disjoint node sets $\mathcal{V}_1$ and $\mathcal{V}_2$, unless
they are known in advance.
We do this by identifying the number of eigenvalues that are approximately
zero.

In principle, this could be done by detecting how many eigenvalues have
absolute value larger than a fixed tolerance, but this process is extremely
sensitive to the choice of the tolerance.
In our numerical experimentation, we found it to be more reliable to detect the largest
gap between ``small'' and ``large'' eigenvalues.

To do this, we compute the ratios
\begin{equation}\label{rhoi}
\rho_i = \frac{|\lambda_{i+1}|}{|\lambda_{i}|}, \quad i=1,2,\ldots,n-1.
\end{equation}
Then, for suitably chosen constants $R$ and $\tau$, we consider the index set
\begin{equation}\label{setJ}
\mathcal{J} = \left\{ i\in \{1,2,\ldots,n-1\}:\ \rho_i>R \text{ and }
|\lambda_{i+1}|>\tau \right\}.
\end{equation}
In our experiments, we set $R=10^2$ and $\tau=10^{-8}$.

An index $i$ is in $\mathcal{J}$ if there is a significant gap between
$\lambda_i$ and $\lambda_{i+1}$ ($\rho_i>R$), and $\lambda_{i+1}$ is
numerically nonzero ($|\lambda_{i+1}|>\tau$).
If the set $\mathcal{J}$ is empty, then we are not able to identify a partition of
the nodes, and we consider the cardinality of the sets $\mathcal{V}_1$ and
$\mathcal{V}_2$ to be the same.
On the contrary, we let $k$ be the index defined by
$$
\rho_k = \max_{i\in\mathcal{J}} \rho_i,
$$
and set
$$
n_2 = \left\lceil \frac{n-k}{2} \right\rceil, \quad n_1 = n-n_2,
$$
where $\left\lceil x\right\rceil$ denotes the closest integer to the real
number $x$.

The above approach is clearly not completely robust. It is easy to trick it by
constructing particular numerical examples, for example by letting $C$ in
\eqref{adjbip} have singular values that decay to zero exponentially, or by
introducing large gaps in the spectrum of the adjacency matrix.
Nevertheless, we found the procedure quite accurate on networks stemming from
real-world applications; see, e.g., Figures~\ref{yeeig} and~\ref{geeig}
in Section~\ref{sec5}.

In order to avoid overflow, it may be preferable to use the reciprocal
ratios $\rho_i^{-1}$. This is not required in our Matlab implementation, given
the features of the programming language.

\item
The subsequent step is to find the sets $\mathcal{V}_1$ and
$\mathcal{V}_2$, and reorder the nodes.
Assume that $\mathcal{G}$ is bipartite, but that the adjacency matrix 
$A$ corresponds to a random ordering of the nodes, so that
$$
A = \Pi A_B \Pi^T,
$$
for a permutation matrix $\Pi$ and a matrix $A_B$ of the form \eqref{adjbip}.
%Obviously, the structure \eqref{adjbip} is lost
In this case, the spectral factorization \eqref{spefac} becomes
$$
A = (\Pi Q) \mathcal{D} (\Pi Q)^T,
$$
i.e., the rows of the eigenvector matrix are permuted. In order to recover the structure 
of the eigenvectors, let us partition the eigenvector matrix as in
\begin{equation}\label{Wpart}
W := \Pi Q = \begin{bmatrix} W_1 & W_2 & W_3 \end{bmatrix},
\end{equation}
with $W_1,W_3\in\R^{n\times n_2}$ and $W_2\in\R^{n\times (n_1-n_2)}$.

Assume first that $n_1>n_2\geq 1$ and consider the matrix block $W_2$.
For \eqref{eigstructg} to be valid, the last $n_2$ 
rows of $W_2$ must vanish.  Sorting in descending order the 1-norms of its rows 
concentrates the smallest entries in the lower block of $W_2$. 
Applying the corresponding permutation $\sigma$ to the rows of $W$ brings this
matrix to the form \eqref{eigstructg} and the adjacency 
matrix to the form \eqref{adjbip}, with the block $C$ possibly permuted.
When $n_1=n_2$ the block $W_2$ is empty, so we consider
the matrix $W_1-W_3Z$. As its first $n_1$ rows should be exactly
zero, we sort the 1-norms of its rows in ascending order,
and apply the obtained permutation $\sigma$ to the rows of $W$.
After the reordering, the first $n_1$ nodes are in the set $\mathcal{V}_1$, and
the remaining $n_2$ are in the set $\mathcal{V}_2$.
We note that applying the permutation $\sigma$ to the rows and columns of
the initial adjacency matrix $A$ highlights the presence in the graph of an
approximate bipartite structure.

\item
To finally obtain an approximation of the matrix \eqref{adjbip} by the computed
spectral factorization, we first approximate the eigenvector matrix $W_B$ by
solving problem \eqref{minW}, and then approximate the 
eigenvalues in \eqref{Aspecg} by scalars that appear in $\pm$ 
pairs using Proposition \ref{theor2g}. Specifically, we let the $\alpha_j$ in the 
proposition be the eigenvalues \eqref{Alambdag}. The $\beta_j$ defined in the 
proposition are the eigenvalues of the matrix $D$ in \eqref{spectrg}, in the same order.

\end{enumerate}

\begin{algorithm}[!ht]
\begin{algorithmic}[1]
\REQUIRE adjacency matrix $A$ of size $n$, the user may optionally provide the
cardinalities $n_1$ and $n_2$ of $\cV_1$ and $\cV_2$
\ENSURE permutation $\sigma$ which reorders the nodes, adjacency matrix $A_B$
	of the approximated bipartite graph
\STATE compute the spectral factorization $A=W\mathcal{D}W^T$, with
	$\lambda_1\geq\cdots\geq\lambda_n$
\\ \COMMENT{\textbf{Step 1 of the algorithm}}
\IF{$n_1$, $n_2$ are not provided}
	\STATE sort the eigenvalues by increasing absolute value
	\STATE compute $\rho_i$, $i=1,\ldots,n-1$ by \eqref{rhoi}
	\STATE construct set $\mathcal{J}$ by \eqref{setJ}
	\IF{$\mathcal{J}=\emptyset$}
		\STATE $n_1=\left\lceil n/2\right\rceil$, $n_2=n-n_1$
	\ELSE
		\STATE $k = \arg\max_{i\in\mathcal{J}} \rho_i$
		\STATE $n_2 = \left\lceil (n-k)/2 \right\rceil$, $n_1 = n-n_2$
	\ENDIF
\ENDIF
\\ \COMMENT{\textbf{Step 2 of the algorithm}}
\STATE partition $W = \begin{bmatrix} W_1 & W_2 & W_3 \end{bmatrix}$ as in
	\eqref{Wpart}
\IF{$n_1>n_2$}
	\STATE find the permutation $\sigma$ which sorts the 1-norms of the
		rows of $W_2$ decreasingly
\ELSE
	\STATE find the permutation $\sigma$ which sorts the 1-norms of the
		rows of $W_1-W_3Z$ increasingly
\ENDIF
\STATE apply the permutation $\sigma$ to the rows of $W$
\\ \COMMENT{\textbf{Step 3 of the algorithm}}
\STATE approximate the eigenvectors of $A$ by minimizing \eqref{minW}
\STATE approximate the eigenvalues of $A$ by $\pm$ pairs $\beta_i$, by
	Proposition \ref{theor2g}
\STATE set $D=\diag(\beta_1,\ldots,\beta_n)$
\STATE construct the adjacency matrix $A_B=WDW^T$ of the bipartite graph
\end{algorithmic}
\caption{Spectral bipartization algorithm.}
\label{alg:specbip}
\end{algorithm}

The above procedure, outlined in Algorithm~\ref{alg:specbip}, determines 
the eigenvectors and eigenvalues of a 
matrix $A_B$ with the block structure
\begin{equation}\label{matrg}
A_B = \begin{bmatrix}
O & C\\
C^T & O
\end{bmatrix},
\end{equation}
where the matrix $C$ has real entries.
The matrix $A_B$ may have a different number of nonzero entries than $A$.
In fact, not all nonzero entries may be positive. We can handle this issue
in several ways:
\begin{itemize}
\item
Allow $A_B$ to be an adjacency matrix for a weighted graph with both positive and negative 
weights.
\item
Allow $A_B$ to be an adjacency matrix for a weighted graph with positive weights. We 
achieve this by replacing the matrix $C$ in \eqref{matrg} by the closest matrix, 
$C_+$, in the Frobenius norm with nonnegative entries. The matrix $C_+$ is
obtained from $C$ by setting all negative entries to zero.
\item
Require $A_B$ to represent an unweighted graph. The closest such matrix in the Frobenius 
norm to the matrix \eqref{matrg} is obtained by setting every entry of $C$ to the
closest member of the set $\{0,1\}$.
\end{itemize}
The last procedure is the one adopted in the numerical experiments presented
in Section \ref{sec5}. 

Algorithm~\ref{alg:specbip} can be applied only to small to medium sized
problems, i.e., when it is possible to compute a full spectral factorization of
$A$.
For larger problems, one may reduce the complexity of the computation by
renouncing the third step of the algorithm.
Indeed, when $n_1-n_2$ is not too large, a partial spectral factorization may
lead to constructing a basis for the null space of $A$, that is, to obtaining
the matrix $W_2$.
This would allow one to generate the permutation $\sigma$ that takes the
adjacency matrix to an almost bipartite form, identifying the two sets $\cV_1$
and $\cV_2$.

\section{Anti-communities}\label{sec3}

Let us consider a symmetric matrix $A$
of size $n=n_1+n_2$ with a zero leading square block of size $n_1$.
Then, $A$ may be considered the adjacency matrix of a network with an
anti-community of $n_1$ nodes.
The matrix has the form
\begin{equation}\label{antiadj}
A = \begin{bmatrix}
O_{n_1} & C\\
C^T & B
\end{bmatrix},
\end{equation}
with $C$ of size $n_1\times n_2$ and $B$ a square matrix of order $n_2$.
In the following, we denote by $\mathcal{N}(C)$ the null space of $C$, by
$\mathcal{R}(C)$ its range, and by $B_{|\mathcal{N}(C)}$ the restriction of the
submatrix $B$ to $\mathcal{N}(C)$.
\medskip

\begin{theorem}\label{t31}
Let $A$ be as in \eqref{antiadj} and let
$\bm{x}=\left[\begin{smallmatrix}\bm{x}_1 \\ \bm{x}_2 \end{smallmatrix}\right]$
be partitioned consistently with $A$.
Then the equation 
\begin{equation}\label{homogeq}
A\bm{x}=\bm{0}
\end{equation}
has $\nu=\dim{\mathcal{N}(C^T)}$ linearly independent solutions with
$\bm{x}_2=0$. Moreover, if
\[
d = \dim\left(\mathcal{R}(B_{|\mathcal{N}(C)})\cap\mathcal{R}(C^T)\right)\geq 1,
\]
then there are also $d$ linearly independent solutions to $A\bm{x}=\bm{0}$ with
$\bm{x}_2\neq 0$, so that $\dim{\mathcal{N}(A)}=d+\nu$.
\end{theorem}

\begin{proof}
Let $k=\ran(C)$ and consider the case $n_1>n_2=k$.
Let us search for vectors $\bm{x}$ such that $A\bm{x}=\bm{0}$.
Then we have
\begin{equation}\label{splitx}
A\begin{bmatrix}
\bm{x}_1\\
\bm{x}_2
\end{bmatrix} = \begin{bmatrix}
O_{n_1} & C\\
C^T & B
\end{bmatrix}\begin{bmatrix}
\bm{x}_1\\
\bm{x}_2
\end{bmatrix} = \begin{bmatrix}
C\bm{x}_2\\
C^T\bm{x}_1+B\bm{x}_2
\end{bmatrix}.
\end{equation}
Since $C$ is of full rank and $n_1>n_2$, it follows from $C\bm{x}_2=0$ that
$\bm{x}_2=0$ and, hence, $C^T\bm{x}_1=0$. The latter
implies that $\bm{x}_1$ is in the null space of $C^T$, which has dimension
$n_1-n_2$. Thus, the matrix $A$ admits the following linearly independent eigenvectors
corresponding to the eigenvalue $\lambda=0$,
$$
\bm{x}^{(i)} = \begin{bmatrix} \bm{u}_{n_2+i} \\ \bm{0} \end{bmatrix}, \quad
i=1,2,\ldots,n_1-n_2,
$$
where $\bm{u}_i$, $i=1,2,\ldots,n_1$, are the left singular vectors of $C$.
Hence, $\lambda=0$ has multiplicity $n_1-n_2=\dim{\mathcal{N}(C^T)}$.

Let us now assume that $k=n_1<n_2$. Then $A$ may or may not have zero
eigenvalues.
Indeed, for $A$ to have a vanishing eigenvalue, the vector
$\bm{x}_2\in\R^{n_2}$ that appears in \eqref{splitx} has to belong to the null
space of $C$, which has dimension $n_2-n_1$.
Then, there will be zero eigenvalues if and only if the system 
$$
C^T \bm{y} = -B\bm{x}_2
$$
has a solution.

If instead $k=n_1=n_2$, i.e., if $C$ is nonsingular,
then $\lambda=0$ implies that both $\bm{x}_1=0$ and $\bm{x}_2=0$.
Hence, $\bm{x} = \left[\begin{smallmatrix} \bm{x}_1 \\ \bm{x}_2
\end{smallmatrix}\right]=0$,
and all the eigenvalues of $A$ are different from zero.

We finally turn to the case when the submatrix $C$ is rank deficient, that is,
$k<\min\{n_1,n_2\}$.
The right-hand side of \eqref{splitx} is equivalent to 
\[
\bm{x}_2\in\mathcal{N}(C),\quad C^T\bm{x}_1=-B\bm{x}_2.
\]
Let $\bm{x}$ be a nontrivial solution of \eqref{homogeq}. When
$\bm{x}_2=\bm{0}$, there has to be a vector $\bm{x}_1\ne\bm{0}$ with
$C^T\bm{x}_1=\bm{0}$. 
Since in this case the null space of $C^T$ has dimension $n_1-k$, there are
$n_1-k$ linearly independent solutions of \eqref{homogeq} with
$\bm{x}_2=\bm{0}$.

The existence of a solution $\bm{x}$ of \eqref{homogeq} with a nonzero subvector
$\bm{x}_2$ is equivalent to
\[
\dim\left(\mathcal{R}(B_{|\mathcal{N}(C)})\cap\mathcal{R}(C^T)\right)\geq 1.
\]
This condition does not hold for most matrix pairs $(B,C)$. 
\end{proof}

\begin{remark}\rm
We note that if $B=0$, then the equation $A\bm{x}=\bm{0}$ has exactly
\[
\dim{\mathcal{N}(C)} + \dim{\mathcal{N}(C^T)} = n-2\ran(C)
\]
linearly independent solutions.
\end{remark}

Theorem \ref{t31} shows that if a network has a large anti-community
($n_1>n_2$), the spectral decomposition $A=WDW^T$ has the form
$$
W = \begin{bmatrix}
E & U_2 & F \\
G & O_{n_2,n_1- n_2} & H
\end{bmatrix}, \quad D = \begin{bmatrix}
D_1 & O & O\\
O & O_{n_1-n_2} & O\\
O & O & D_2
\end{bmatrix}.
$$
The structures of $W$ and $D$ are very similar to those of $W_B$ and $\Lambda_B$
in \eqref{eigstructg}, respectively.
For this reason, the bipartization algorithm described in Section~\ref{sec2.5},
is able to detect the presence of a large anti-community and to order the nodes
so that the adjacency matrix takes the form \eqref{antiadj}.
In case a group of nodes is only approximately an anti-community, the algorithm
produces an adjacency matrix that approximates \eqref{antiadj}.

To summarize, when $n_1>n_2$, if a network is either bipartite or contains a
large anti-community, its adjacency matrix has zero eigenvalues;
the converse is not true.
If $A$ has a multiple zero eigenvalue, then we can
recognize the presence of one of the two above cases by observing the structure
of the eigenvector matrix.
%We can also measure the distance from being bipartite or from having a large
%anti-community. 
%and determine a network that satisfies one of the two conditions.

\section{Computed examples}\label{sec5}

In the following numerical experiments, we fix the integers $n_1$ and $n_2$, and construct
a random matrix $A$ of the form \eqref{adjbip}, with a sparse block $C$ with density $\xi$.
The matrix is first perturbed, by replacing its (1,1) and (2,2) blocks by sparse matrices 
of appropriate size and density $\eta$, and then ``scrambled'', by applying the same random 
permutation to its rows and columns.

We apply the algorithm of Section~\ref{sec2} to the matrix $A$ either by
supplying the cardinality of the two sets $\mathcal{V}_1$ and $\mathcal{V}_2$
(this approach is referred to as \texttt{specbip}-$n$), or letting the method
estimate $n_1$ and $n_2$ from the data; we refer to the latter approach as
\texttt{specbip}.
Since the block (1,2) of the matrix returned by the method is generally
permuted with respect to the initial test matrix, the rows and columns are
reordered according to the original sequence of the nodes.
The final reordering allows us to compare the resulting matrix $A_B$ to the test
matrix $A$.

Our results are compared to the ones obtained by red-black ordering using the
MatlabBGL library \cite{MatBGL}, a Matlab package implementing
graph algorithms. A matrix has a red-black ordering if the corresponding graph
is bipartite. To find a bipartite ordering, this software uses a breadth first
search algorithm, starting from an arbitrary vertex. The partition of the nodes
is determined by forming a group containing all the vertices having even
distance from the root, and another group with the vertices at odd distance
from the root. This procedure is designed to bipartite networks, not to produce
an approximation when the bipartization is not exact.

\begin{figure}[hbt]
\begin{center}
\includegraphics[width=.32\textwidth]{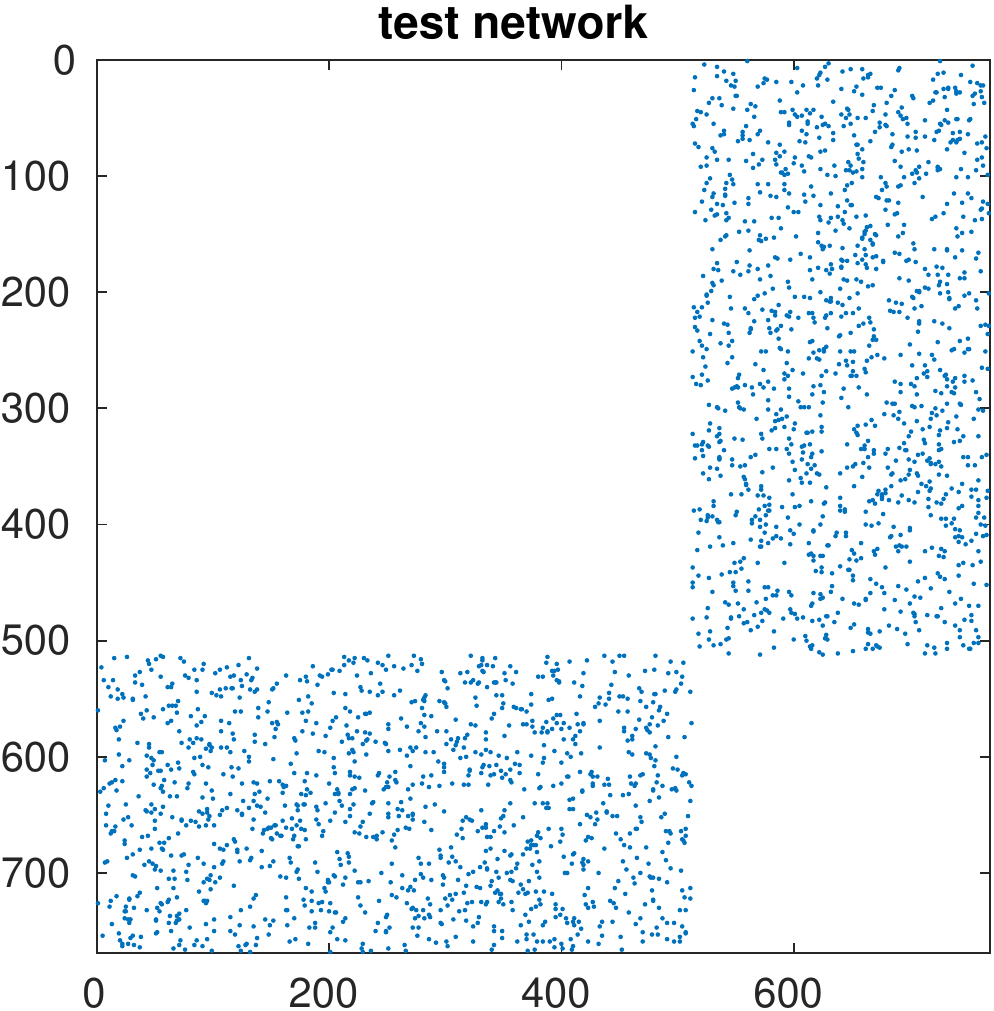}\hfill
\includegraphics[width=.32\textwidth]{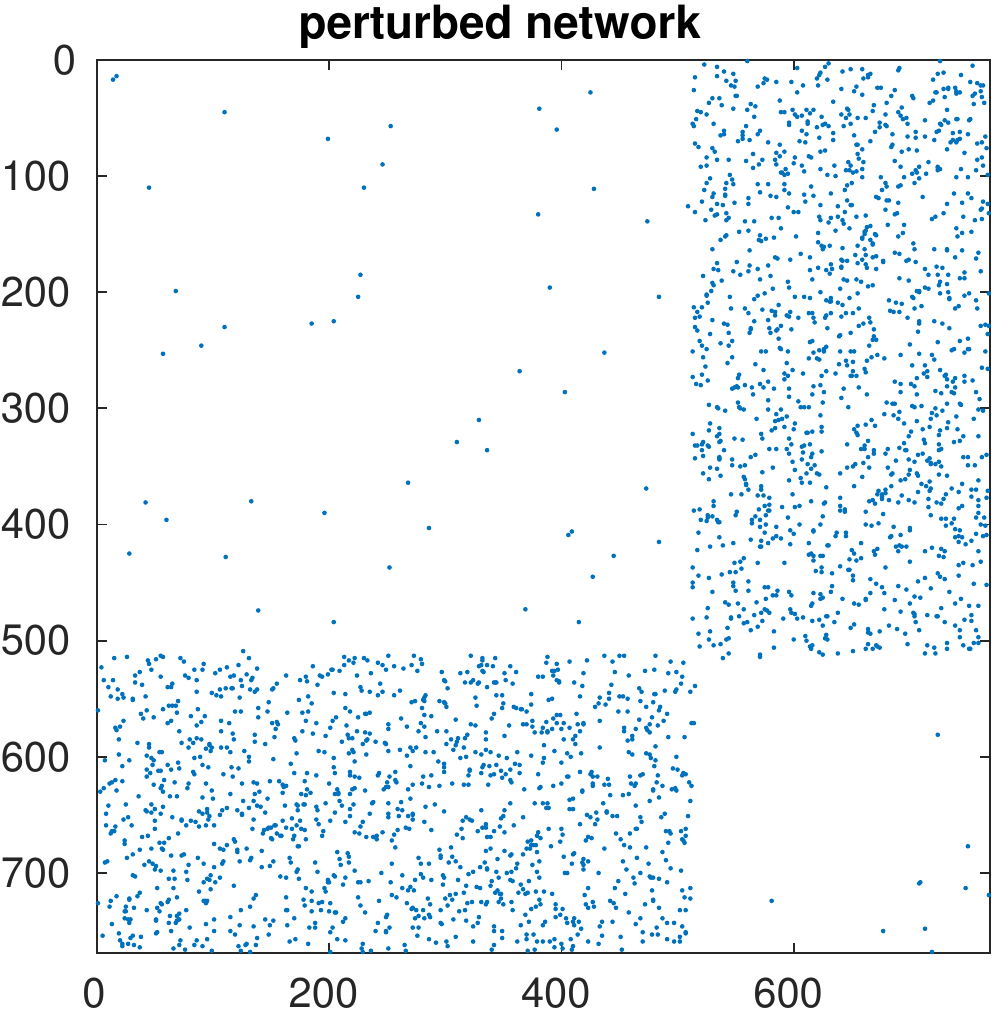}\hfill
\includegraphics[width=.32\textwidth]{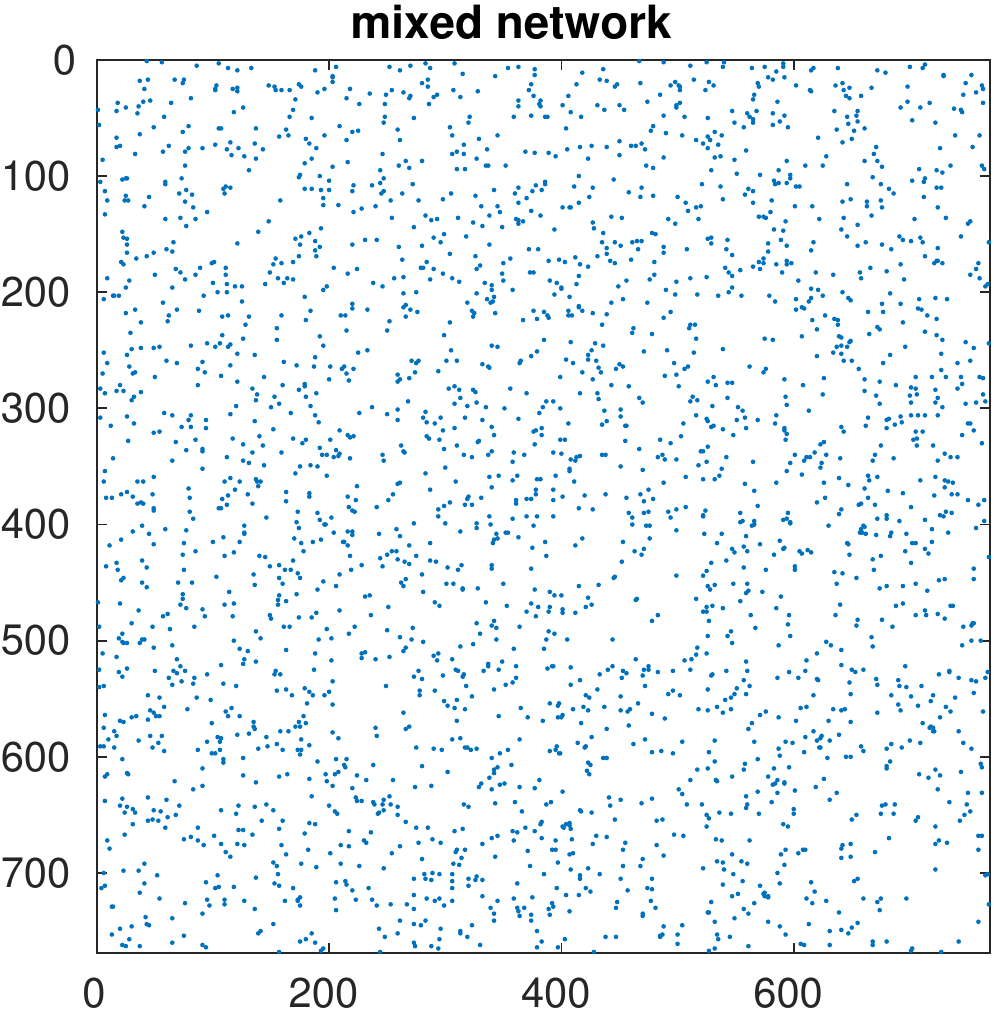}\hfill
\includegraphics[width=.32\textwidth]{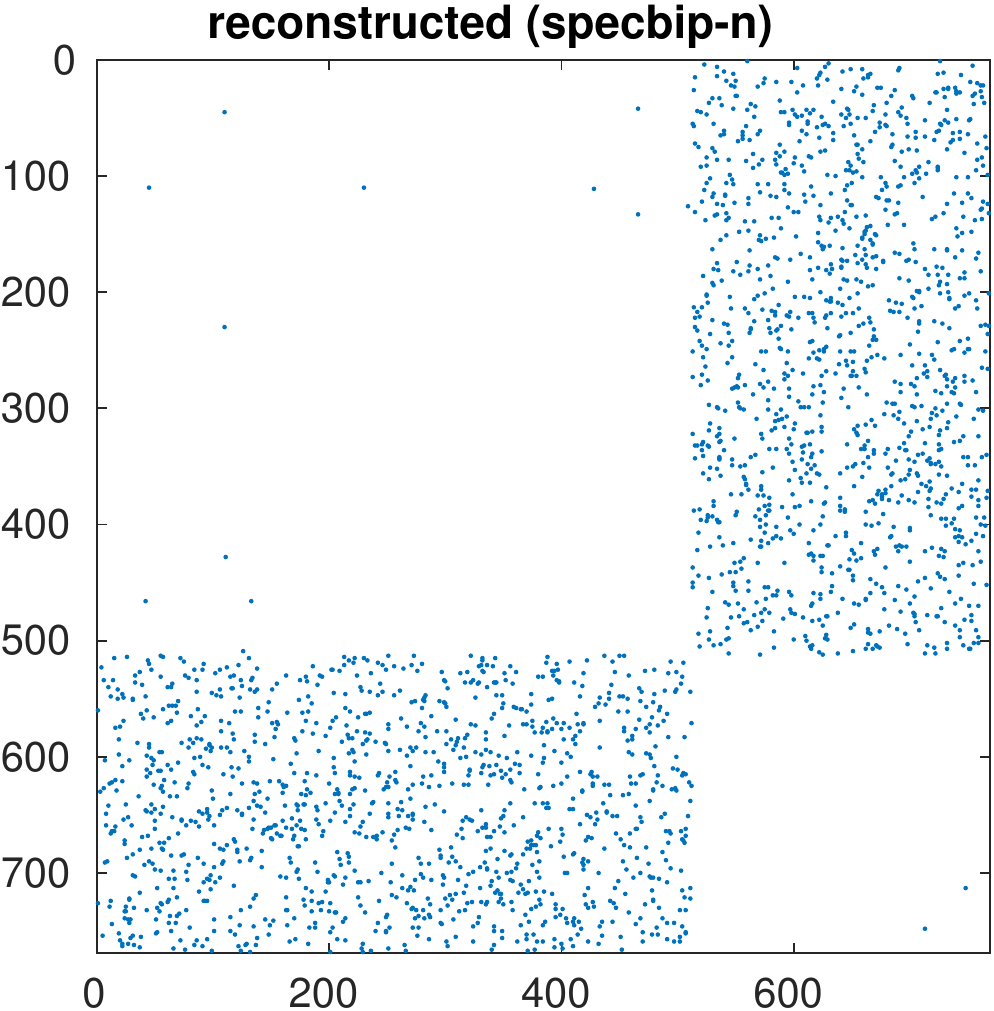}\hfill
\includegraphics[width=.32\textwidth]{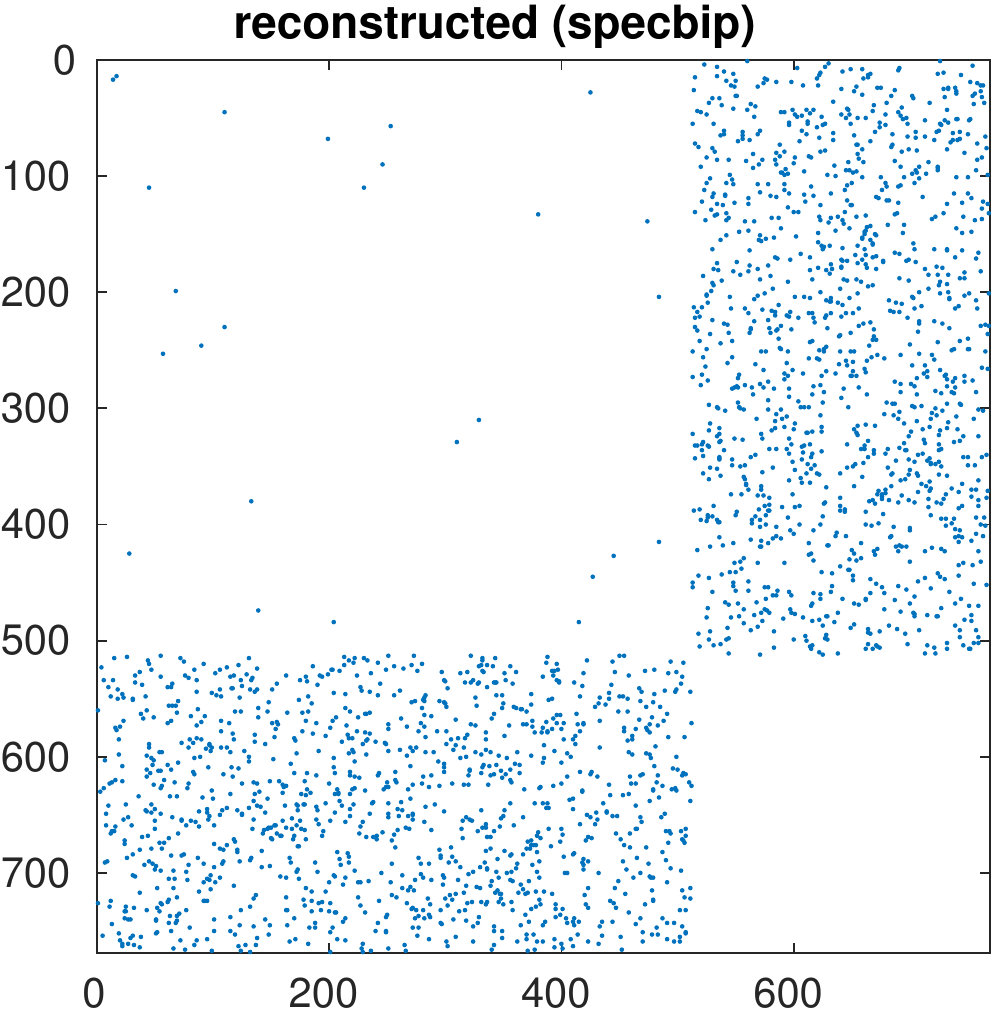}\hfill
\includegraphics[width=.32\textwidth]{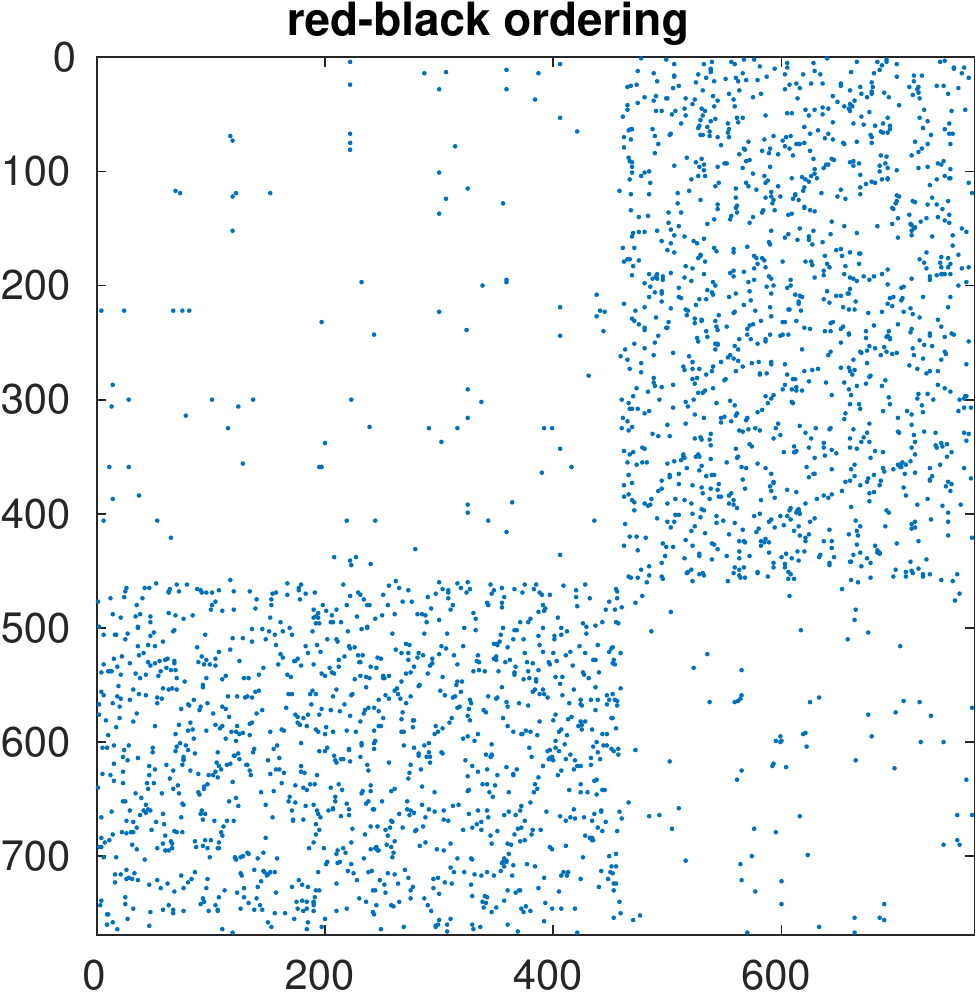}\hfill
\caption{$(n_1,n_2)=(512,256)$, $(\tilde{n}_1,\tilde{n}_2)=(492,276)$,
$\xi=10^{-2}$, $\eta=10^{-4}$.}
\label{ex1fig}
\end{center}
\end{figure}

\begin{figure}[hbt]
\begin{center}
\includegraphics[width=.6\textwidth]{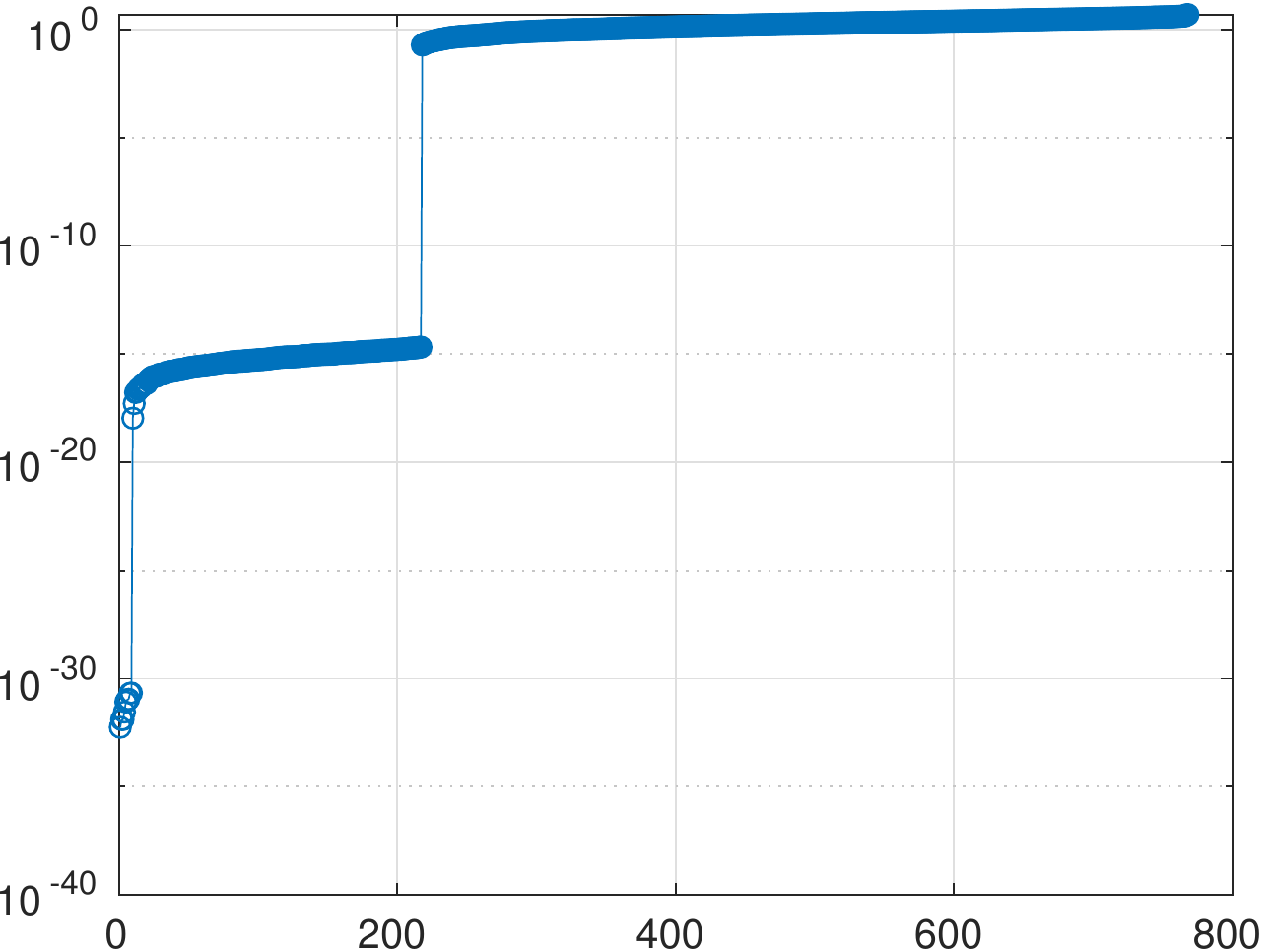}
\caption{$(n_1,n_2)=(512,256)$, $(\tilde{n}_1,\tilde{n}_2)=(492,276)$,
$\xi=10^{-2}$, $\eta=10^{-4}$.}
\label{ex1eig}
\end{center}
\end{figure}

Figure~\ref{ex1fig} displays the results for a test matrix with
$(n_1,n_2)=(512,256)$, sparsity $\xi=10^{-2}$, and perturbation
$\eta=10^{-4}$.
In particular, it reports in the upper row a spy plot of the original test
matrix, the perturbed version, with random arcs in the (1,1) and (2,2) blocks,
and the permuted matrix that is fed to the bipartization methods.
The bottom row shows the reconstructed networks.
The \texttt{specbip}-$n_1$ approach, which receives the information about the
cardinality of the node sets, produces the matrix closest to the original.
The general algorithm estimates the cardinalities
$(\tilde{n}_1,\tilde{n}_2)=(492,276)$, according to the number of ``small''
eigenvalues; see Figure~\ref{ex1eig}, where the absolute values
of the eigenvalues are displayed in nondecreasing order.
This algorithm produces a slightly less accurate approximation than the
previous one, which is anyway much better than the matrix produced by the
red/black ordering.

\begin{figure}[hbt]
\begin{center}
\includegraphics[width=.32\textwidth]{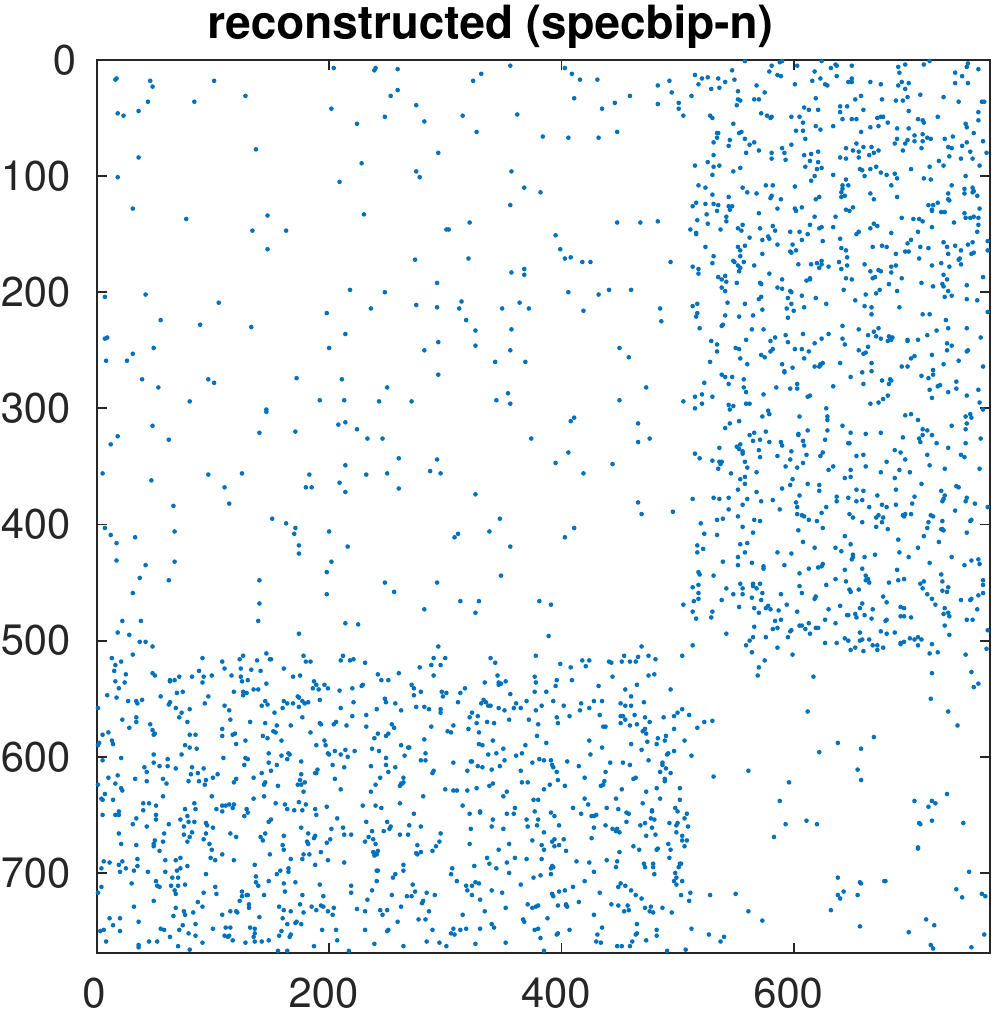}\hfill
\includegraphics[width=.32\textwidth]{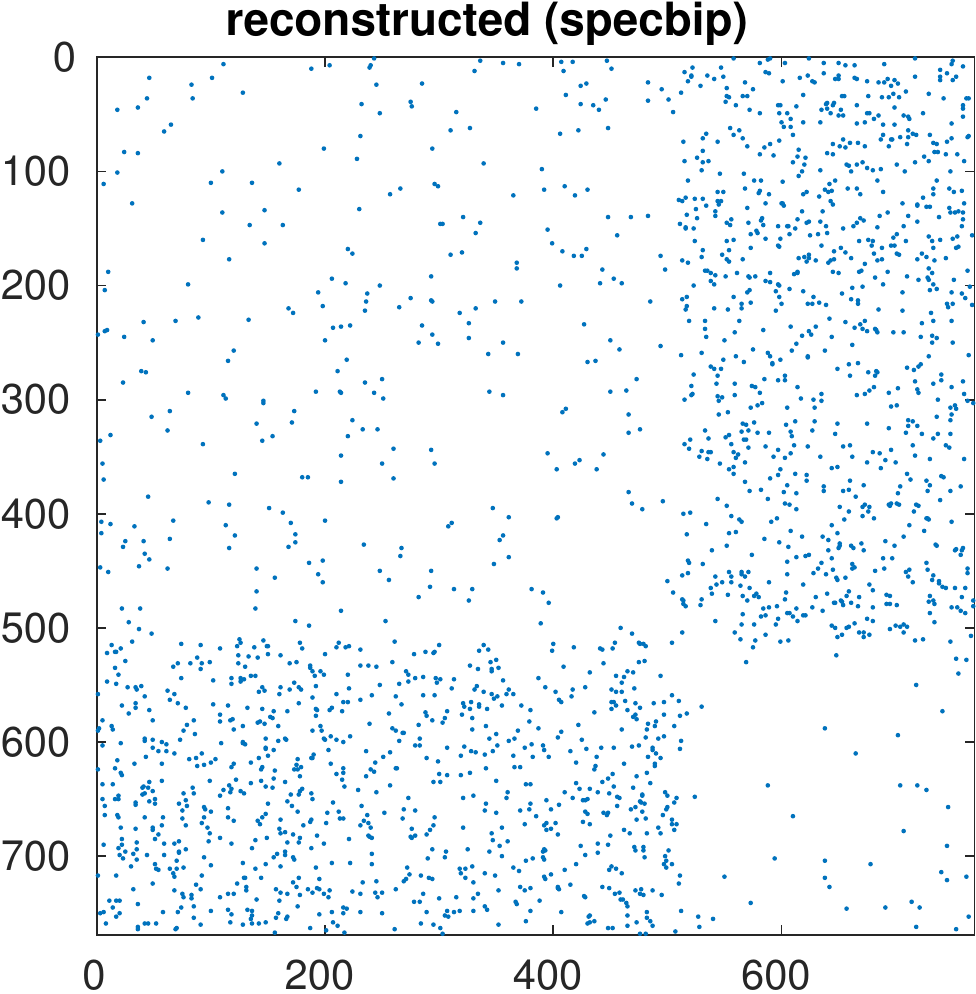}\hfill
\includegraphics[width=.32\textwidth]{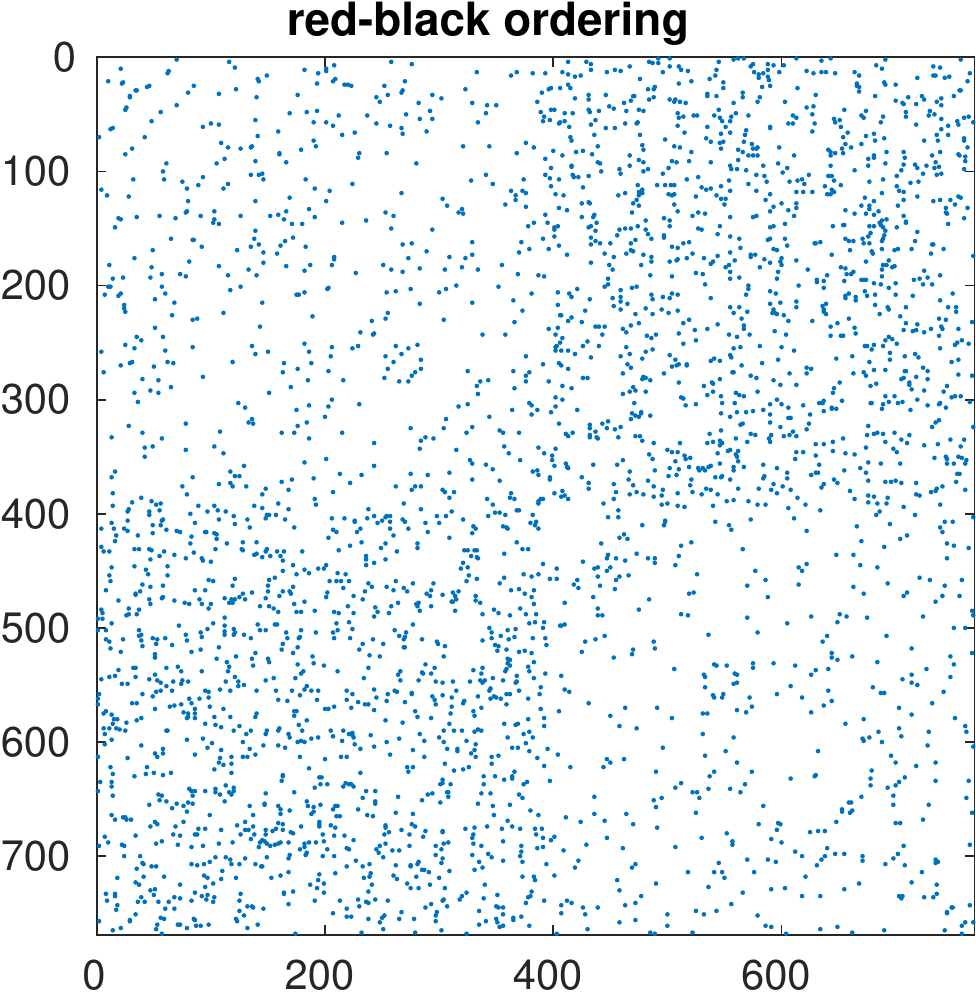}\hfill
\caption{$(n_1,n_2)=(512,256)$, $(\tilde{n}_1,\tilde{n}_2)=(396,372)$,
$\xi=10^{-2}$, $\eta=10^{-3}$.}
\label{ex2fig}
\end{center}
\end{figure}

Figure~\ref{ex2fig} shows the results for a test matrix similar to the previous
one, but with a larger perturbation $\eta=10^{-3}$.
The estimation of $(n_1,n_2)$ is inaccurate, but the approximation produced by
the \texttt{specbib} methods is quite close to the unperturbed matrix, while
the red/black ordering matrix is far from it.

Now, let 
$$
E = A-A_B = \begin{bmatrix} E_{11} & E_{12} \\ E_{21} & E_{22} \end{bmatrix},
$$
where $E_{11}$ and $E_{22}$ are square matrices of size $n_1$ and $n_2$,
respectively, and let $|M|$ denote the number of nonzero elements of
$M$.
To evaluate the quality of the results, we consider the following three indices
$$
\mathcal{I}_B = 1-b_s, \qquad
\mathcal{E}_B = \frac{|E_{11}|}{n_1^2} + \frac{|E_{22}|}{n_2^2}, \qquad
\mathcal{E}_A = \frac{|E_{12}|}{|C|}.
$$
The first two indices measure the distance of $A_B$ from the adjacency matrix
of a bipartite graph; see \eqref{bipind} for the definition of $b_s$.
The third index measures the approximation error with respect to the starting
bipartite network \eqref{adjbip}.
To better evaluate the error in the bipartition, we introduce the fourth
index $\mathcal{E}_{N}=\widetilde{\mathcal{E}}_{N}/n_1$, where
$\widetilde{\mathcal{E}}_{N}$ is the number of nodes from the set $\cV_1$ that
were incorrectly ascribed to the set $\cV_2$.

Tables \ref{tab24}, \ref{tab25}, and \ref{tab14} report the average
values of the above four quality indices over
10 realizations of the random test networks. Three different pairs $(n_1,n_2)$
are considered; each table refers to different densities $\xi$ and $\eta$; $T$ stands for the execution time in seconds.

\begin{table}[ht!]
\centering
\caption{Results for $\xi=10^{-2}$, $\eta=10^{-4}$.}
\begin{tabular}{cccc}
\hline
(256,128) & \texttt{specbip}-$n_1$ & \texttt{specbip} & red-black \\
\hline
$\mathcal{I}_B$ & 1.22e-16 & 1.89e-16 & 2.33e-03 \\
$\mathcal{E}_B$ & 5.46e-04 & 6.74e-04 & 3.72e-03 \\
$\mathcal{E}_A$ & 2.80e-01 & 2.79e-01 & - \\
$\mathcal{E}_N$ & 1.45e-01 & 1.58e-01 & 2.76e-01 \\
$T$ & 4.94e-02 & 5.05e-02 & 3.15e-04 \\
\hline
(512,256) & \texttt{specbip}-$n_1$ & \texttt{specbip} & red-black \\
\hline
$\mathcal{I}_B$ & 1.11e-17 & 1.11e-17 & 2.98e-03 \\
$\mathcal{E}_B$ & 1.13e-04 & 1.50e-04 & 3.39e-03 \\
$\mathcal{E}_A$ & 4.84e-02 & 6.27e-02 & - \\
$\mathcal{E}_N$ & 3.36e-02 & 5.96e-02 & 2.97e-01 \\
$T$ & 2.77e-01 & 2.95e-01 & 4.94e-04 \\
\hline
(1024,512) & \texttt{specbip}-$n_1$ & \texttt{specbip} & red-black \\
\hline
$\mathcal{I}_B$ & 7.77e-17 & 0.00e+00 & 4.17e-02 \\
$\mathcal{E}_B$ & 9.92e-05 & 2.11e-04 & 4.75e-03 \\
$\mathcal{E}_A$ & 1.06e-01 & 1.80e-01 & - \\
$\mathcal{E}_N$ & 3.62e-02 & 1.15e-01 & 2.75e-01 \\
$T$ & 1.92e+00 & 1.94e+00 & 8.67e-04 \\
\hline
\end{tabular}
\label{tab24}
\end{table}

\begin{table}[ht!]
\centering
\caption{Results for $\xi=10^{-2}$, $\eta=10^{-5}$.}
\begin{tabular}{cccc}
\hline
(256,128) & \texttt{specbip}-$n_1$ & \texttt{specbip} & red-black \\
\hline
$\mathcal{I}_B$ & 1.11e-17 & 7.77e-17 & 1.68e-06 \\
$\mathcal{E}_B$ & 6.68e-04 & 8.79e-04 & 3.36e-03 \\
$\mathcal{E}_A$ & 2.70e-01 & 2.68e-01 & - \\
$\mathcal{E}_N$ & 1.23e-01 & 1.49e-01 & 2.58e-01 \\
$T$ & 4.39e-02 & 4.70e-02 & 1.01e-03 \\
\hline
(512,256) & \texttt{specbip}-$n_1$ & \texttt{specbip} & red-black \\
\hline
$\mathcal{I}_B$ & 0.00e+00 & 2.22e-17 & 1.40e-04 \\
$\mathcal{E}_B$ & 3.05e-05 & 1.91e-05 & 8.81e-04 \\
$\mathcal{E}_A$ & 3.88e-02 & 2.38e-02 & - \\
$\mathcal{E}_N$ & 1.87e-02 & 1.93e-02 & 3.16e-01 \\
$T$ & 2.72e-01 & 2.77e-01 & 5.12e-04 \\
\hline
(1024,512) & \texttt{specbip}-$n_1$ & \texttt{specbip} & red-black \\
\hline
$\mathcal{I}_B$ & 0.00e+00 & 0.00e+00 & 4.04e-03 \\
$\mathcal{E}_B$ & 1.91e-07 & 1.03e-05 & 1.07e-03 \\
$\mathcal{E}_A$ & 1.73e-04 & 9.49e-03 & - \\
$\mathcal{E}_N$ & 9.77e-05 & 9.47e-03 & 3.25e-01 \\
$T$ & 1.91e+00 & 1.89e+00 & 9.52e-04 \\
\hline
\end{tabular}
\label{tab25}
\end{table}

\begin{table}[ht!]
\centering
\caption{Results for $\xi=10^{-1}$, $\eta=10^{-4}$}
\begin{tabular}{cccc}
\hline
(256,128) & \texttt{specbip}-$n_1$ & \texttt{specbip} & red-black \\
\hline
$\mathcal{I}_B$ & 0.00e+00 & 0.00e+00 & 7.71e-02 \\
$\mathcal{E}_B$ & 0.00e+00 & 5.83e-04 & 1.35e-02 \\
$\mathcal{E}_A$ & 2.43e-02 & 4.24e-02 & - \\
$\mathcal{E}_N$ & 0.00e+00 & 2.58e-02 & 3.18e-01 \\
$T$ & 5.56e-02 & 6.05e-02 & 3.07e-03 \\
\hline
(512,256) & \texttt{specbip}-$n_1$ & \texttt{specbip} & red-black \\
\hline
$\mathcal{I}_B$ & 0.00e+00 & 0.00e+00 & 1.44e-01 \\
$\mathcal{E}_B$ & 0.00e+00 & 8.19e-04 & 8.01e-03 \\
$\mathcal{E}_A$ & 8.02e-03 & 5.19e-02 & - \\
$\mathcal{E}_N$ & 0.00e+00 & 4.47e-02 & 3.31e-01 \\
$T$ & 2.77e-01 & 2.76e-01 & 1.08e-03 \\
\hline
(1024,512) & \texttt{specbip}-$n_1$ & \texttt{specbip} & red-black \\
\hline
$\mathcal{I}_B$ & 0.00e+00 & 0.00e+00 & 2.60e-01 \\
$\mathcal{E}_B$ & 0.00e+00 & 1.04e-03 & 6.54e-03 \\
$\mathcal{E}_A$ & 2.33e-03 & 9.04e-02 & - \\
$\mathcal{E}_N$ & 0.00e+00 & 8.71e-02 & 3.28e-01 \\
$T$ & 2.02e+00 & 2.07e+00 & 3.99e-03 \\
\hline
\end{tabular}
\label{tab14}
\end{table}

A comparison of the tables shows that the spectral bipartization algorithm is
always more accurate than the red-black ordering method.
At the same time, it is much slower than the MatlabBGL function, as in our
experiments we compute the whole spectrum of the adjacency matrix, without
exploiting its sparsity.
To be competitive with existing methods for large-scale problems, the spectral
method should be modified in order to perform its task by suitable iterative
methods, in order to take advantage of the structure of the adjacency matrix.

\begin{figure}[htb]
\begin{center}
\includegraphics[width=.48\textwidth]{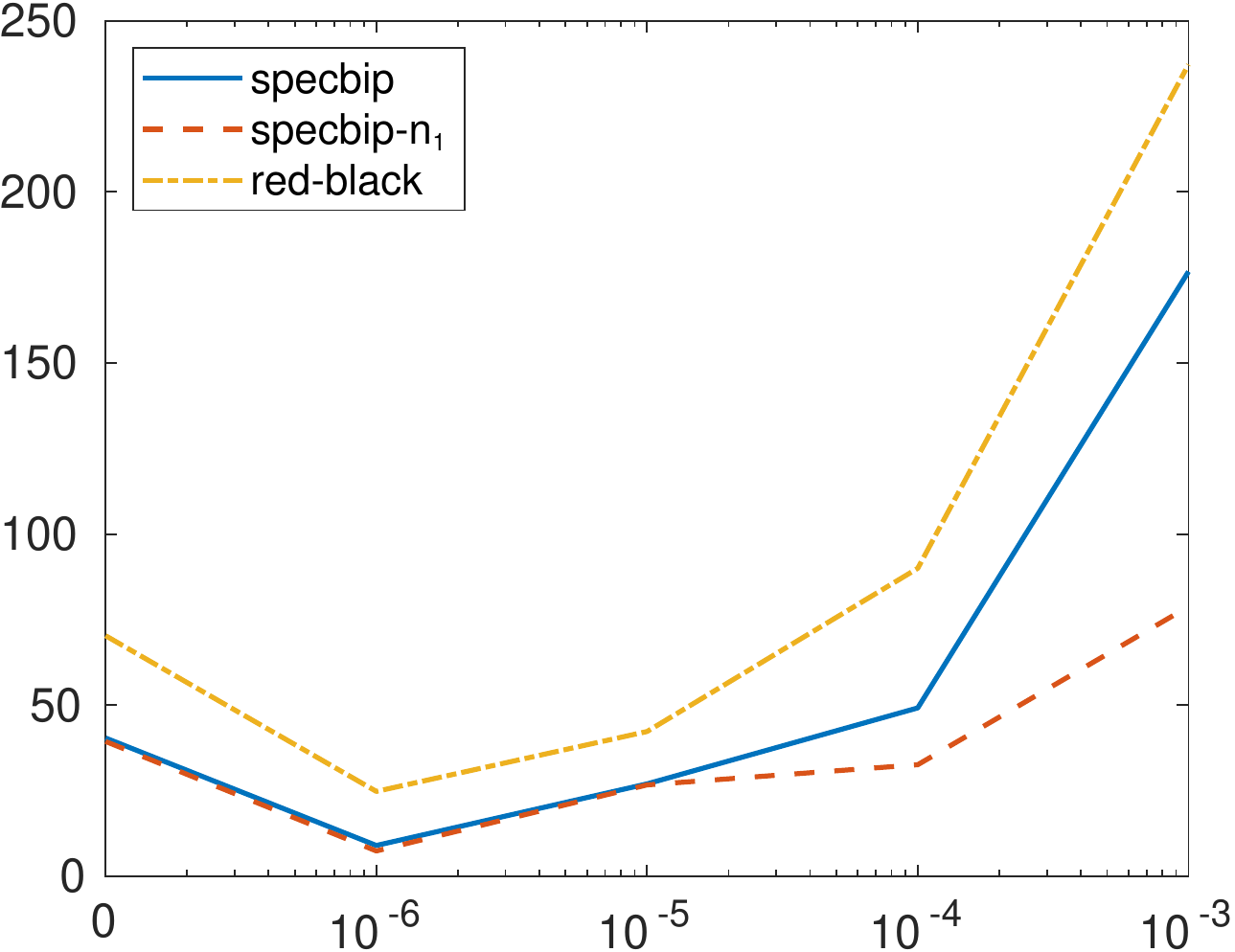}\hfill
\includegraphics[width=.48\textwidth]{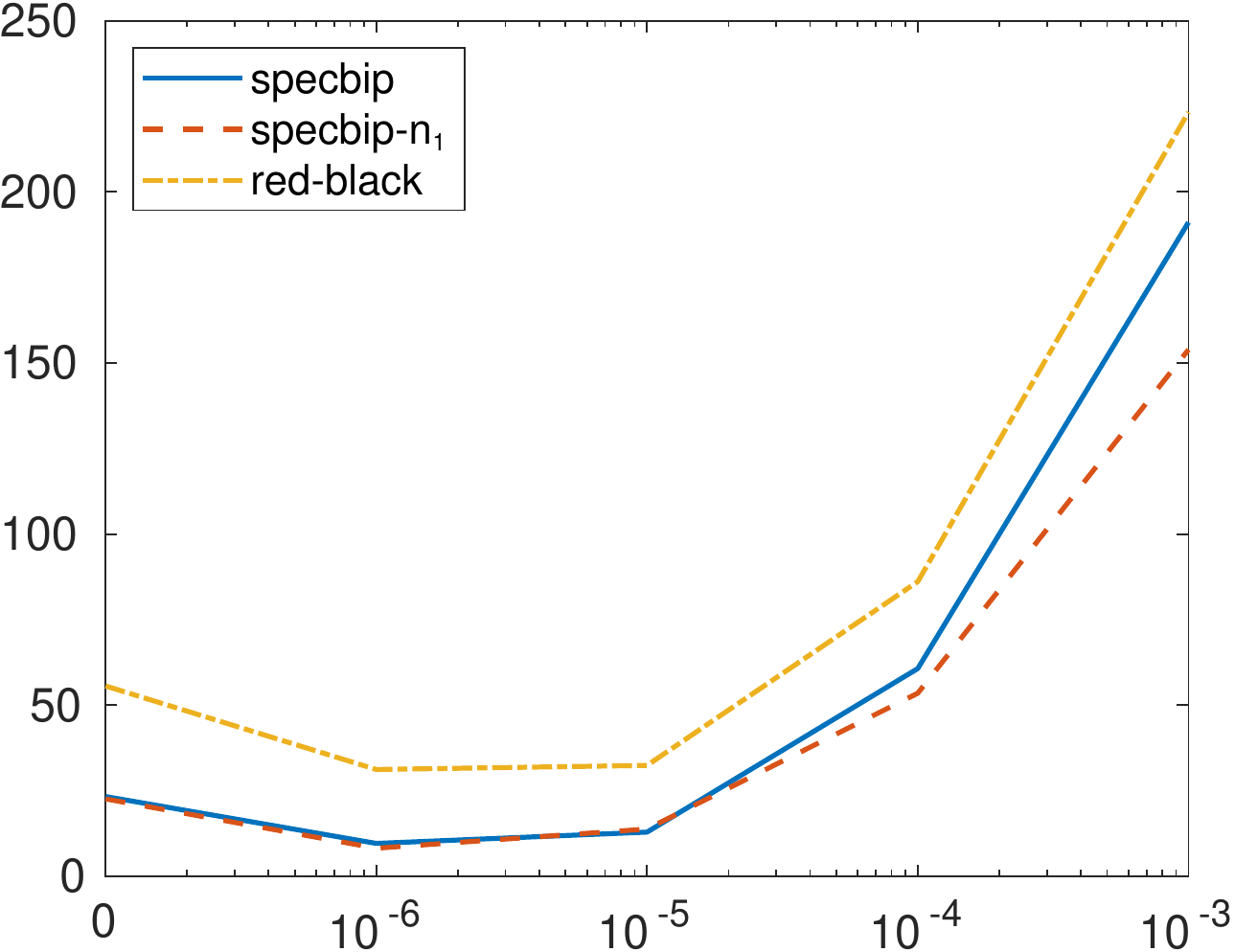}\hfill
\caption{Bipartition error $\widetilde{\mathcal{E}}_N$ for
$(n_1,n_2)=(512,256)$; on the left unweighted random graphs, on the right
weighted random graphs, both with $\xi=10^{-2}$, as a function of 
$\eta=0,10^{-6},10^{-5},\ldots,10^{-3}$.}
\label{fig1}
\end{center}
\end{figure}

From the tables, it can also be observed that knowing in advance the
cardinality of the two sets $\mathcal{V}_1$ and $\mathcal{V}_2$ leads in some
cases to a substantial improvement in the quality of the results.

\begin{figure}[hbt]
\begin{center}
\includegraphics[width=.48\textwidth]{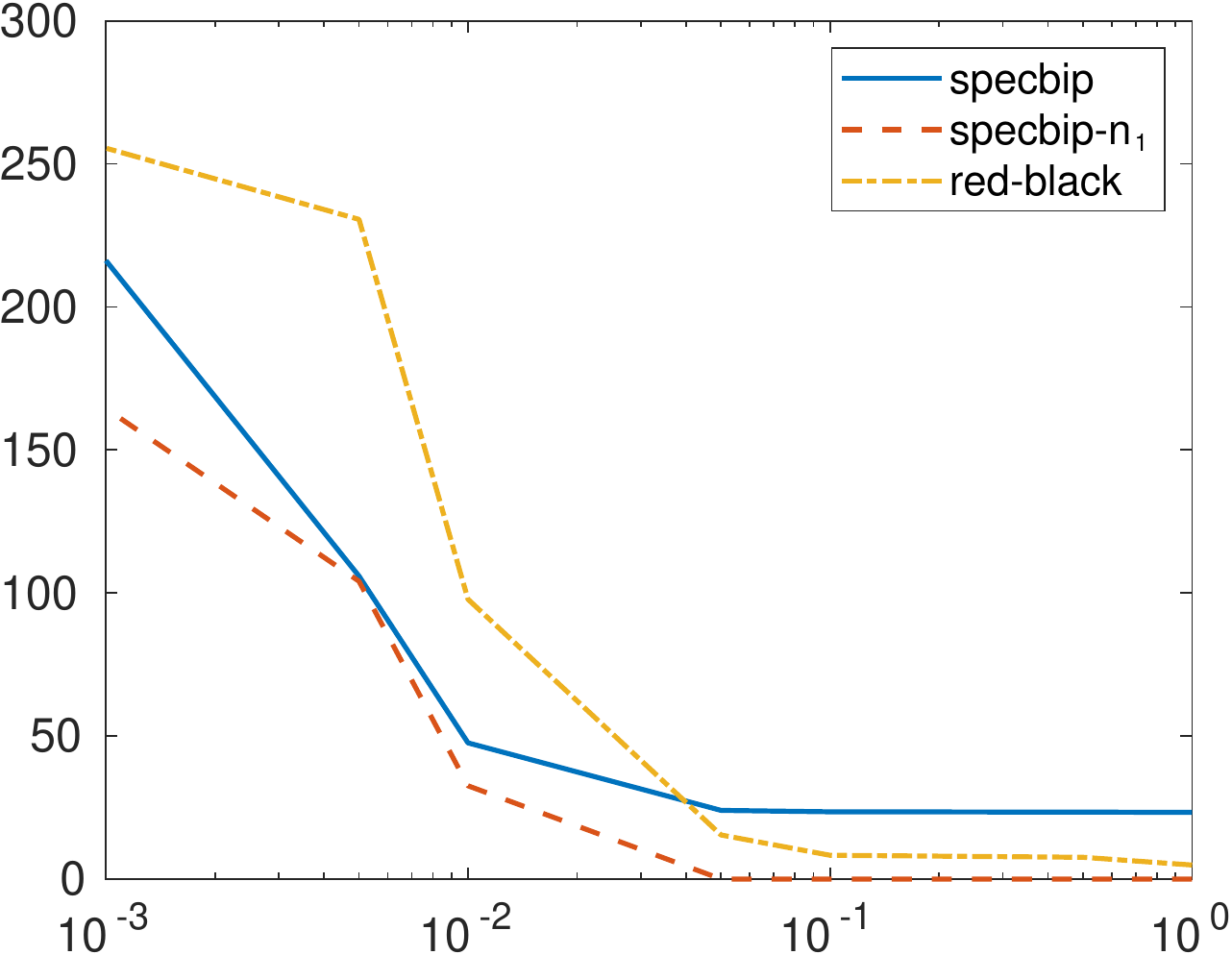}\hfill
\includegraphics[width=.48\textwidth]{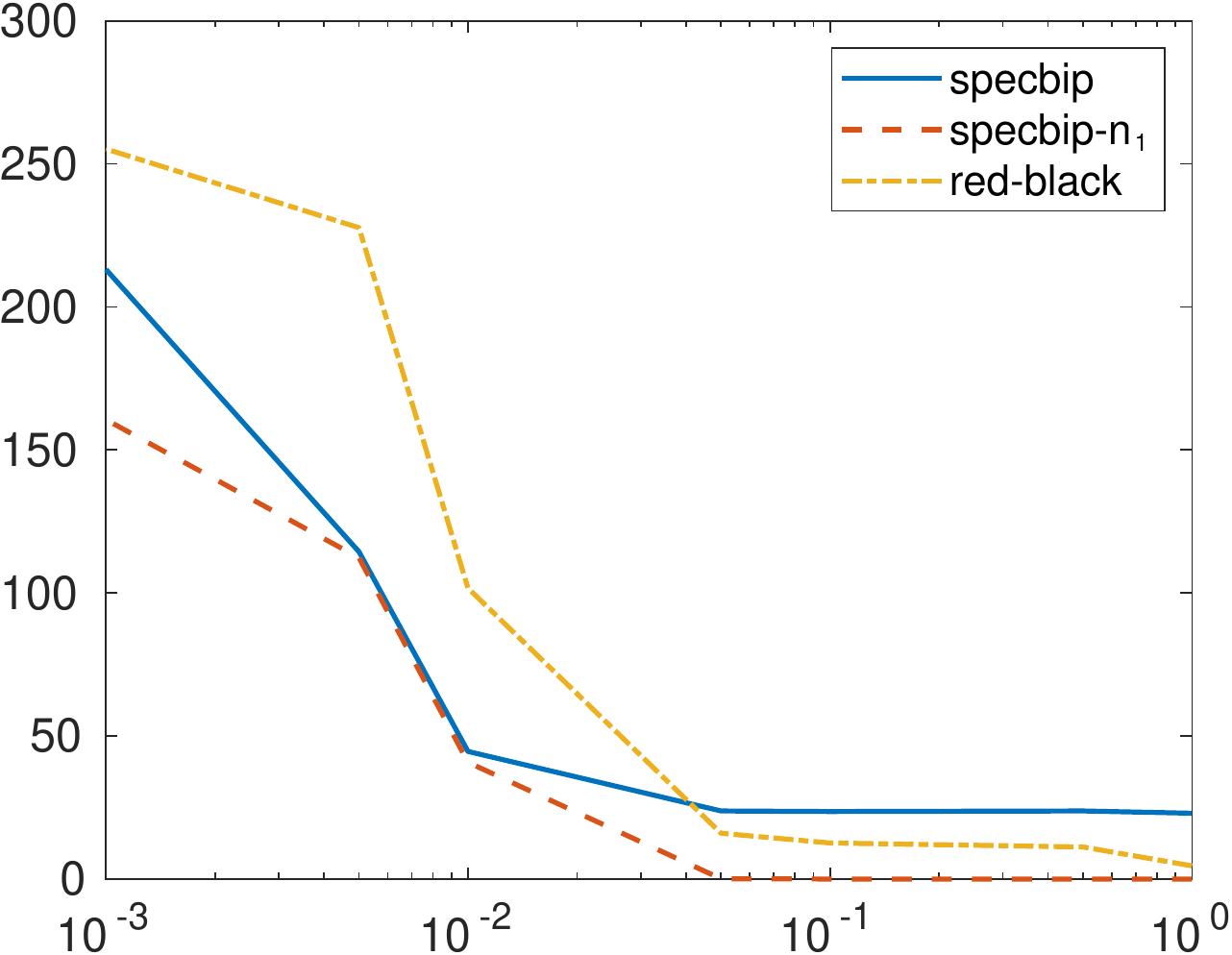}\hfill
\caption{Bipartition error $\widetilde{\mathcal{E}}_N$ for
$(n_1,n_2)=(512,256)$; on the left unweighted random graphs, on the right
weighted random graphs, both with $\eta=10^{-2}$, as a function of 
$\xi=10^{-3},10^{-2},10^{-1},1$.}
\label{fig2}
\end{center}
\end{figure}

To further investigate the behavior of the bipartition error, we construct a
matrix $A$ of the form \eqref{adjbip}, letting $n_1=512$ and $n_2=256$, with a
sparse random block $C$ having density $\xi=10^{-2}$.
After randomly permuting the rows and columns, we apply our algorithms to
this matrix, as well as to those perturbed by replacing the (1,1) and (2,2)
blocks by a sparse matrix with density $\eta=10^{-6},10^{-5},\ldots,10^{-3}$.
The graph on the left of Figure~\ref{fig1} shows the value of the bipartization
error $\mathcal{E}_N$ obtained when the three methods are applied to an
unweighted graph, the one on the right correspond to a weighted graph.
All values are averaged over 10 realizations of the random matrices.
Both graphs show that the bipartization determined by our approaches is closer
to the correct one, with respect to red-black ordering, with
\texttt{specbip}-$n_1$ producing slightly better results. The performance of
all algorithms degrades as the perturbation becomes less sparse.

In Figure~\ref{fig2}, we display the value of $\mathcal{E}_N$ for the same
examples, for a fixed $\eta=10^{-2}$, and letting the density $\xi$ of the
block $C$ take values in $[10^{-3},1]$.
The red-black ordering method is more accurate than the \texttt{specbip}
algorithm for very sparse networks, while providing the correct cardinality of
the set $\mathcal{V}_1$ to \texttt{specbip}-$n_1$ produces the best results.

\subsection{The NDyeast network}

We illustrate the performance of the spectral bipartization algorithm when applied to the
detection of anti-communities by analyzing a case study.
The \emph{NDyeast} network describes the protein interaction network for yeast, each edge representing an interaction between two proteins \cite{jmbo01}.
The data set is available at \cite{webpajek}.
In this section we analyze this network, testing the presence of a
bipartization or of a large anti-community.

The \emph{NDyeast} network has 2114 nodes. There are 74 self-loops (nodes
connected only to themselves) and 268 nodes disconnected from the network.
The adjacency matrix resulting by removing both the self-loops and the
isolated nodes has size $n=1846$, and it has 149 connected components.
They were identified by the \textsf{getconcomp} function from the PQser Matlab
toolbox \cite{CFR}. 

In the case of a reducible adjacency matrix, the spectral bipartization
algorithm should treat each single connected component one at a time. 
Since most of the components in the \emph{NDyeast} network are very small,
often just 2 or 3 nodes, we consider the only component with more than 10
nodes, which has 1458 nodes.
We process the reduced adjacency matrix $A$ with our bipartization method.

The algorithm determines $n_0=564$ zero eigenvalues (see Figure~\ref{yeeig})
and identifies two sets of nodes with cardinalities $n_1=1011$ and $n_2=447$.

\begin{figure}[hbt]
\begin{center}
\includegraphics[width=.6\textwidth]{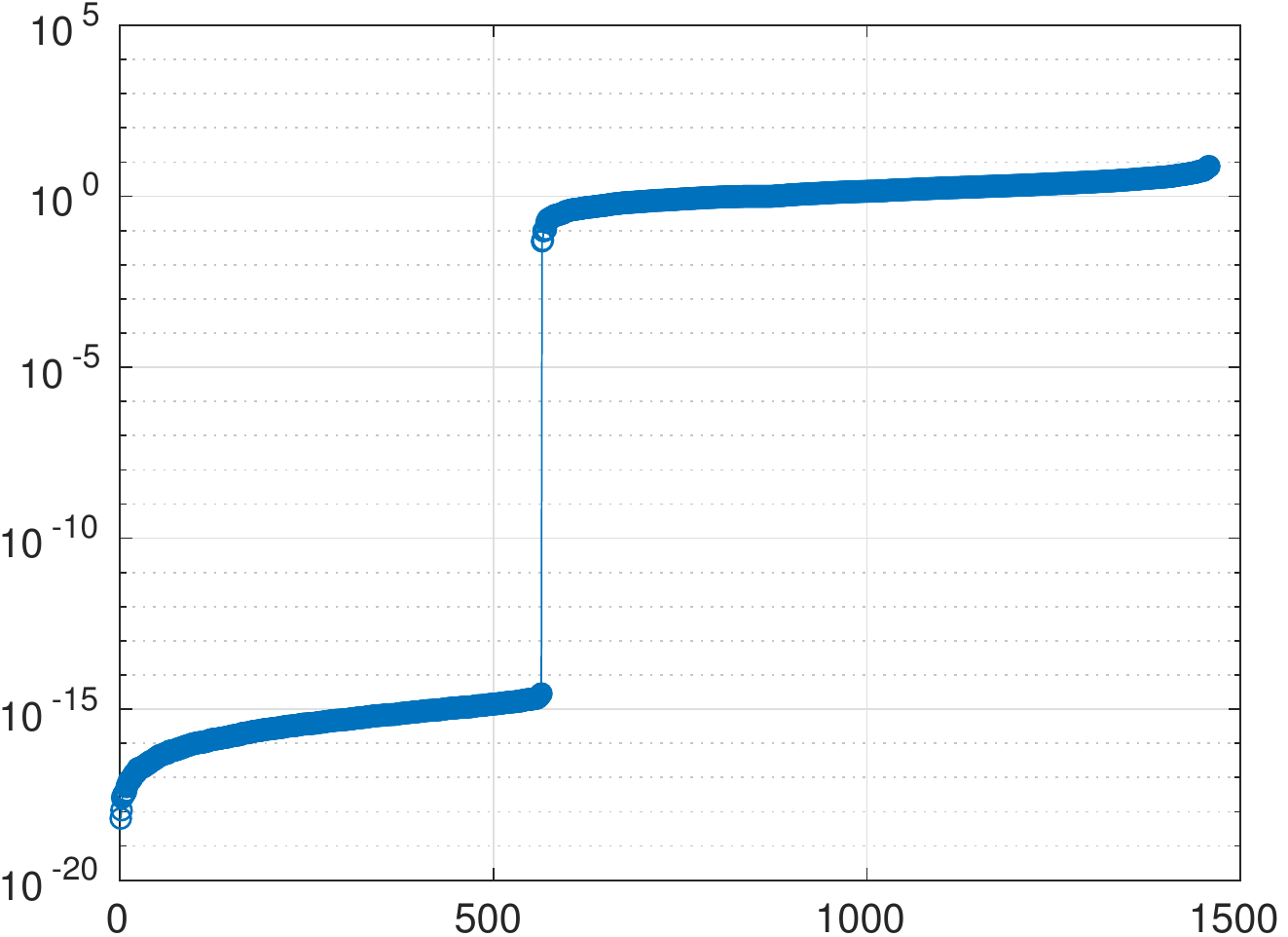}
\caption{Spectrum of the reduced adjacency matrix for the \emph{NDyeast} network.}
\label{yeeig}
\end{center}
\end{figure}

The starting adjacency matrix is displayed in the top-left spy plot of
Figure~\ref{ye1fig}.
The top-right plot shows the same matrix after the ordering produced by the
spectral bipartization algorithm is applied to its rows and columns.
This graph clearly displays that there is a large group of nodes in the
\emph{NDyeast} network that do not communicate much among themselves, that is,
an anti-community.
In the same graph we show the bipartization detected by the algorithm by means
of red lines.

\begin{figure}[hbt]
\begin{center}
\includegraphics[width=.49\textwidth]{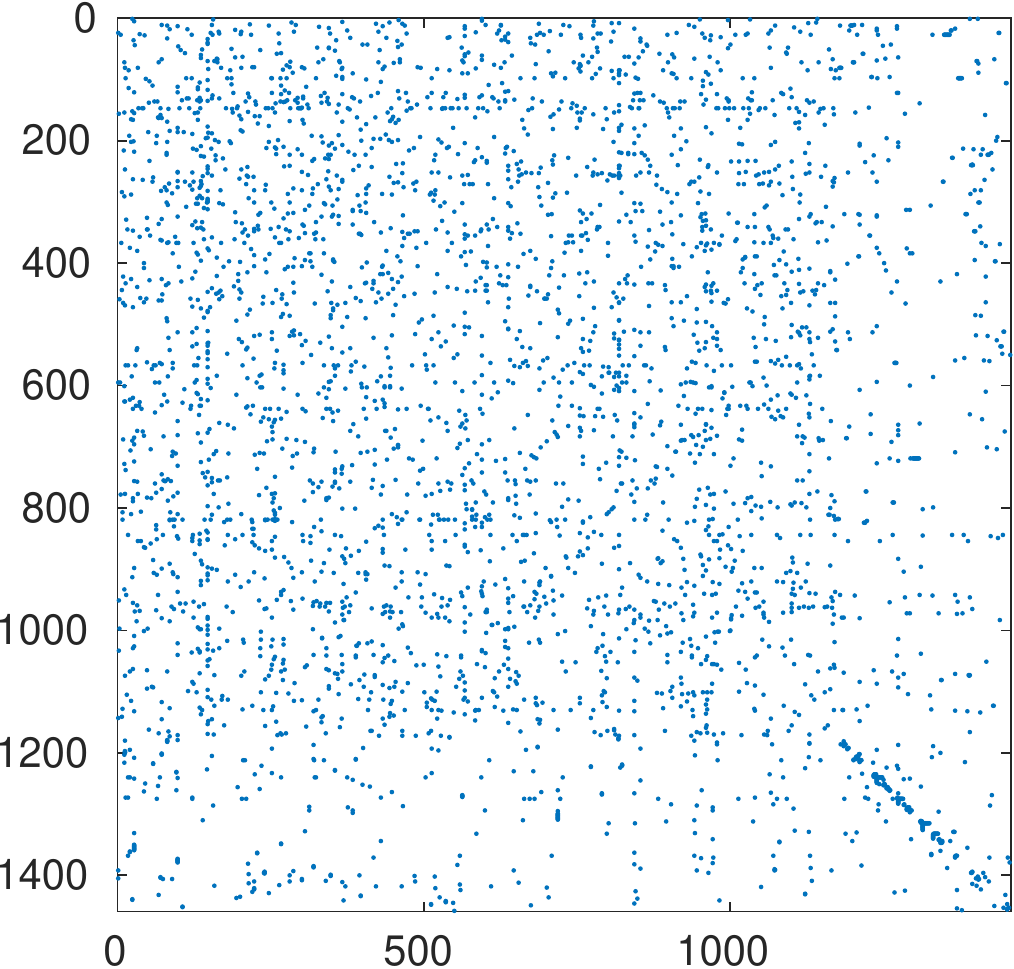}\hfill
\includegraphics[width=.49\textwidth]{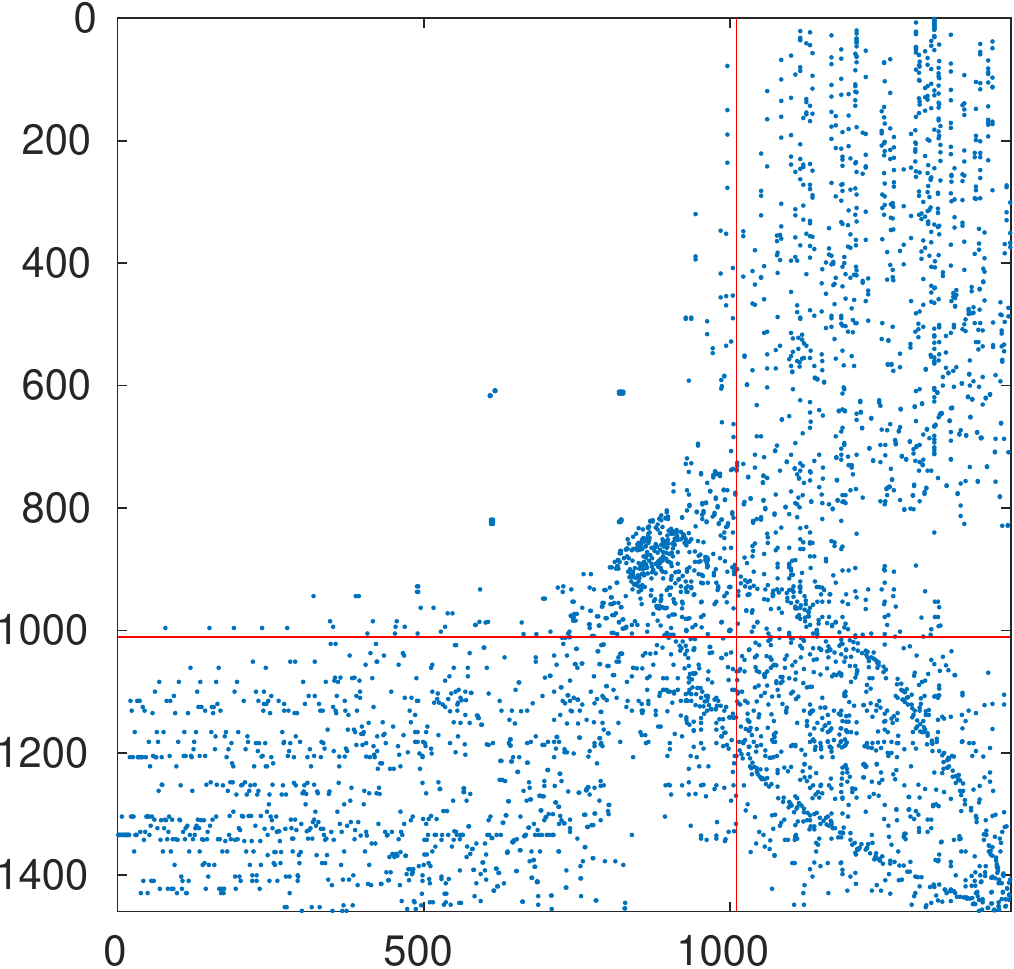}
\includegraphics[width=.49\textwidth]{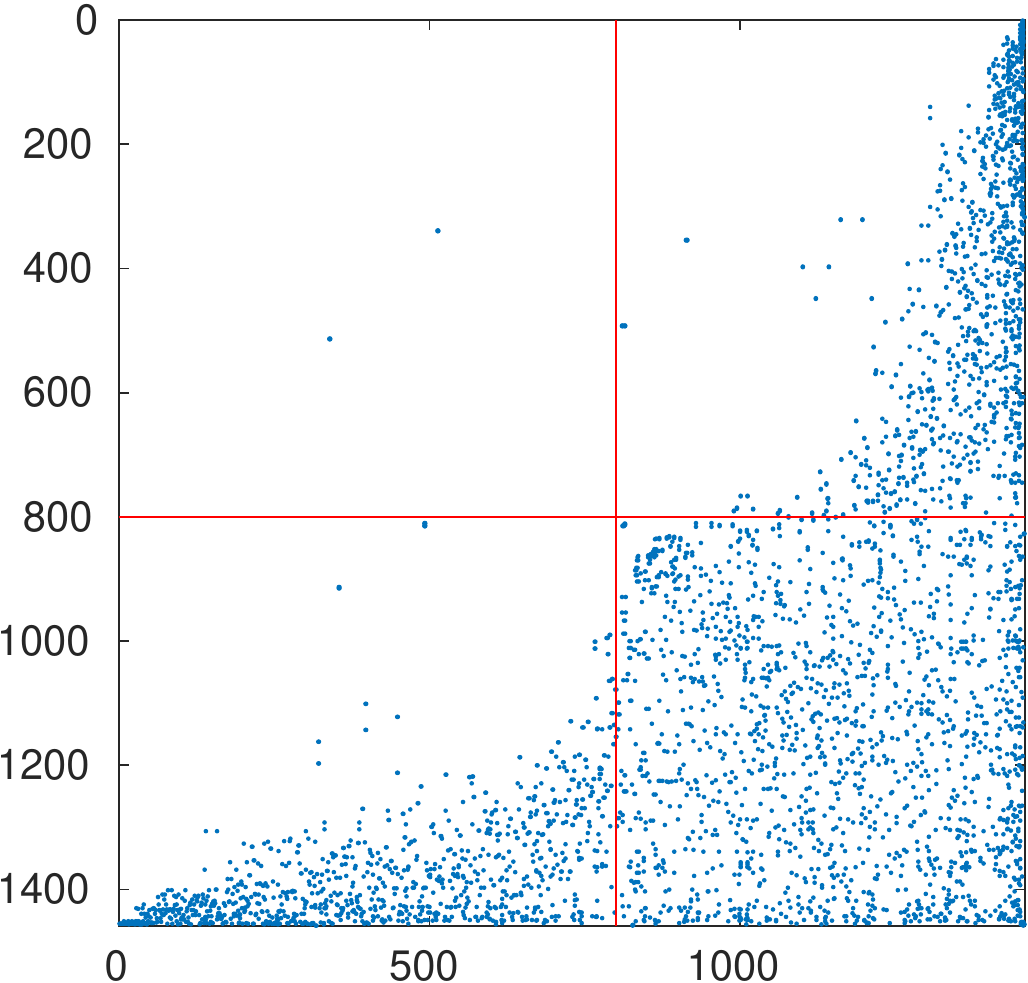}\hfill
\includegraphics[width=.49\textwidth]{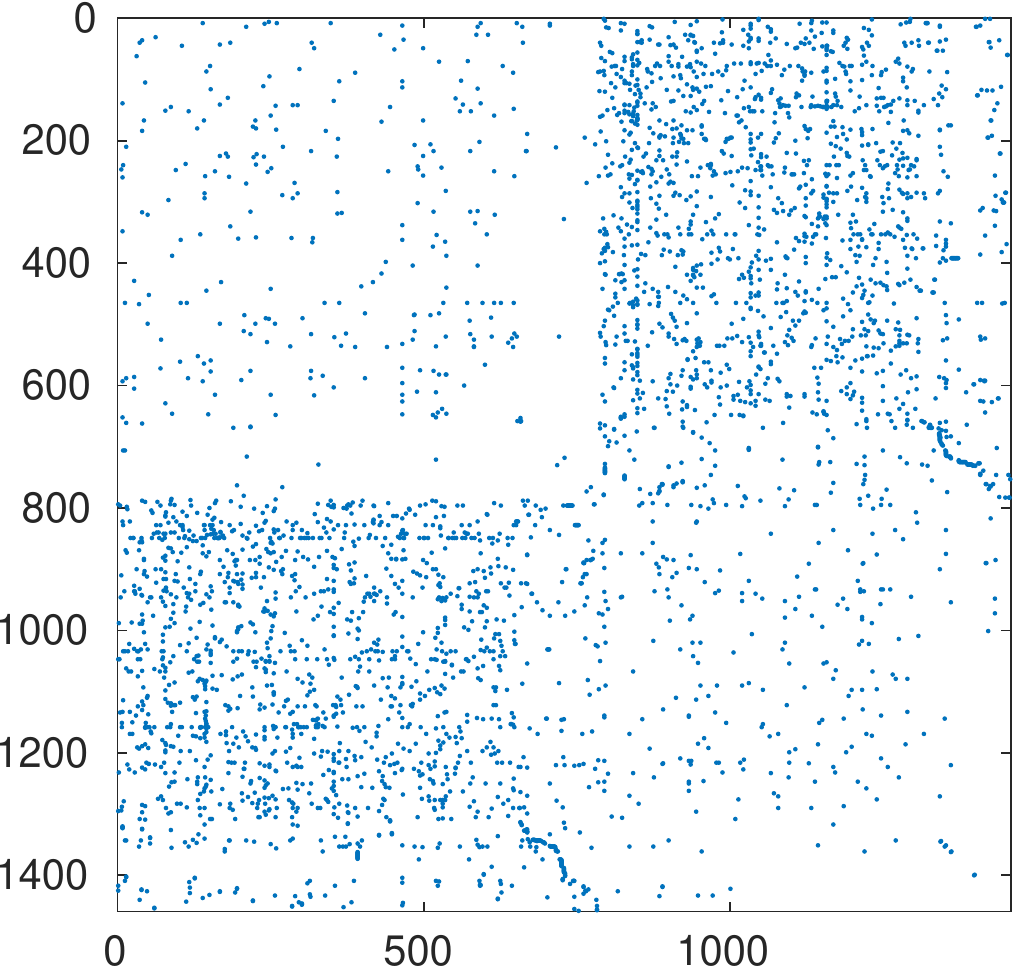}\\
\caption{Analysis of the \emph{NDyeast} network: top-left, the starting
adjacency matrix; top-right, the node reordering produced by the spectral
algorithm; bottom-left, the reordering induced by the choice
$(\tilde{n}_1,\tilde{n}_2)=(800,658)$; bottom-right, the output of the
red-black ordering method.}
\label{ye1fig}
\end{center}
\end{figure}

Our algorithm can also be applied by supplying the values of $(n_1,n_2)$,
rather than estimating them from the number of zero eigenvalues. 
If we do this by setting $\tilde{n}_1=800$ and $\tilde{n}_2=658$, we obtain the
bottom left graph in the same figure. 
It shows that in the group of the first 800 proteins, only four of them
directly interact.

The bottom-right graph of Figure~\ref{ye1fig} displays the result of 
the red-black ordering method, which does not supply any useful information.

We remark that a data set similar to \emph{NDyeast} (but
different) is available at \cite{webpajek}.
It is called simply \emph{yeast}, it consists of 2361 nodes, and it refers to
the paper \cite{slzrc03}.
By processing this data set with our spectral algorithm, we obtain results very
similar to the ones displayed in Figure~\ref{ye1fig}.

\subsection{The geom network}

We also applied the spectral bipartization algorithm to a weighted graph,
namely, the \emph{geom} network,
It is extracted from the Computational Geometry Database \emph{geombib} by B.
Jones (version 2002). Nodes represent authors; the value of the entry $(i,j)$
of the adjacency matrix is the number of papers coauthored by authors $i$ and
$j$. The data set is available at \cite{webpajek}.

The \emph{geom} network has 7343 nodes and 11898 edges. After removing 1185
isolated nodes, the network presents 875 connected components, the largest of
which has 3621 nodes. We applied the bipartization method to the adjacency
matrix associated to this component.

\begin{figure}[hbt]
\begin{center}
\includegraphics[width=.6\textwidth]{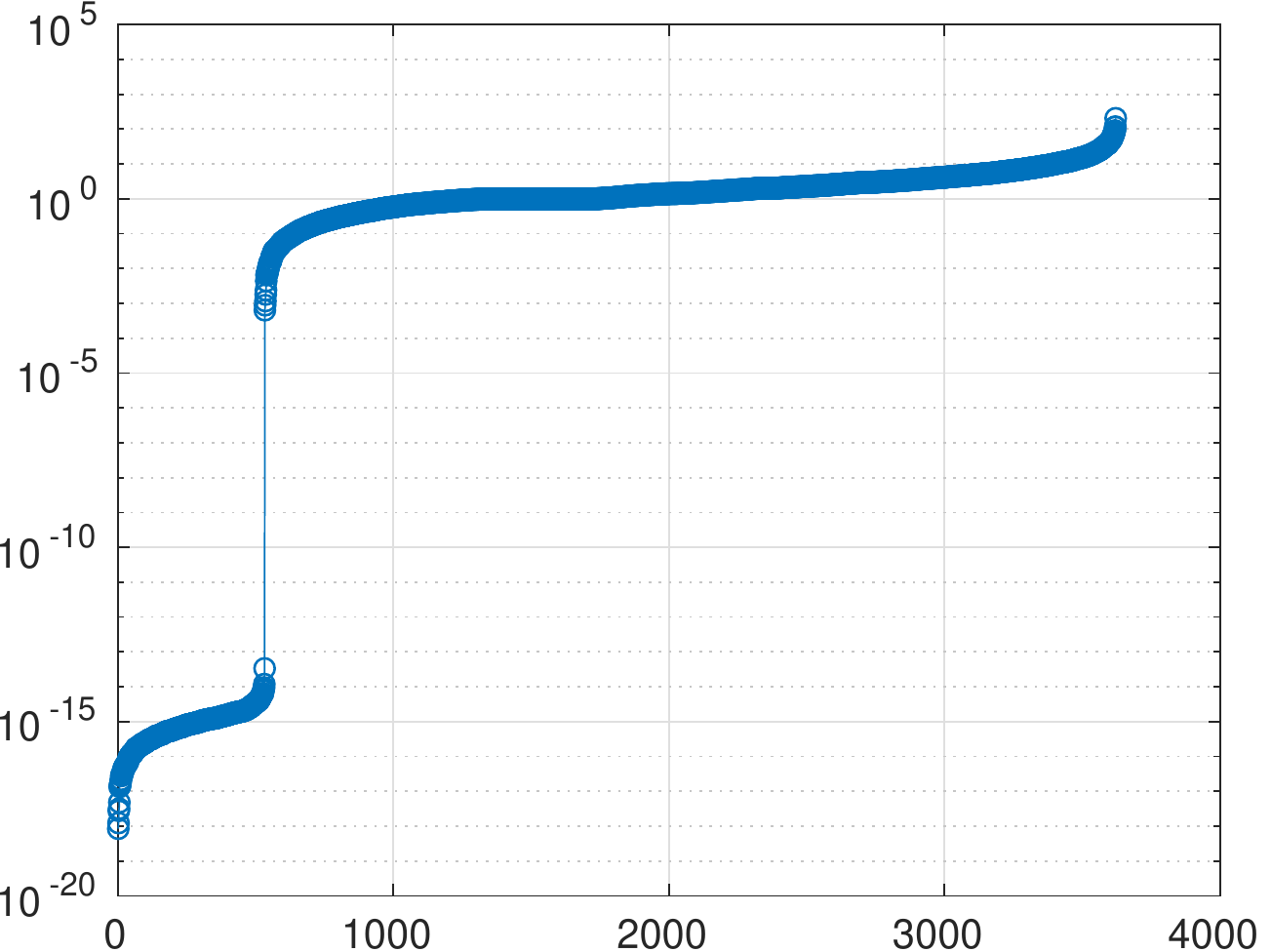}
\caption{Spectrum of the reduced adjacency matrix for the \emph{geom} network.}
\label{geeig}
\end{center}
\end{figure}

The eigenvalues are displayed in Figure~\ref{geeig}: 533 of them are detected
as being numerically zero, and the cardinalities of the two node sets are $n_1=2077$ and
$n_2=1544$.
The left graph of Figure~\ref{ge1fig} reports the spy plot of the original
adjacency matrix; the graph on the right shows the matrix reordered by the
spectral bipartization algorithm.
The graph highlights the presence of an anti-community of about 1000 authors,
who did not collaborate with each other when writing papers.

\begin{figure}[hbt]
\begin{center}
\includegraphics[width=.49\textwidth]{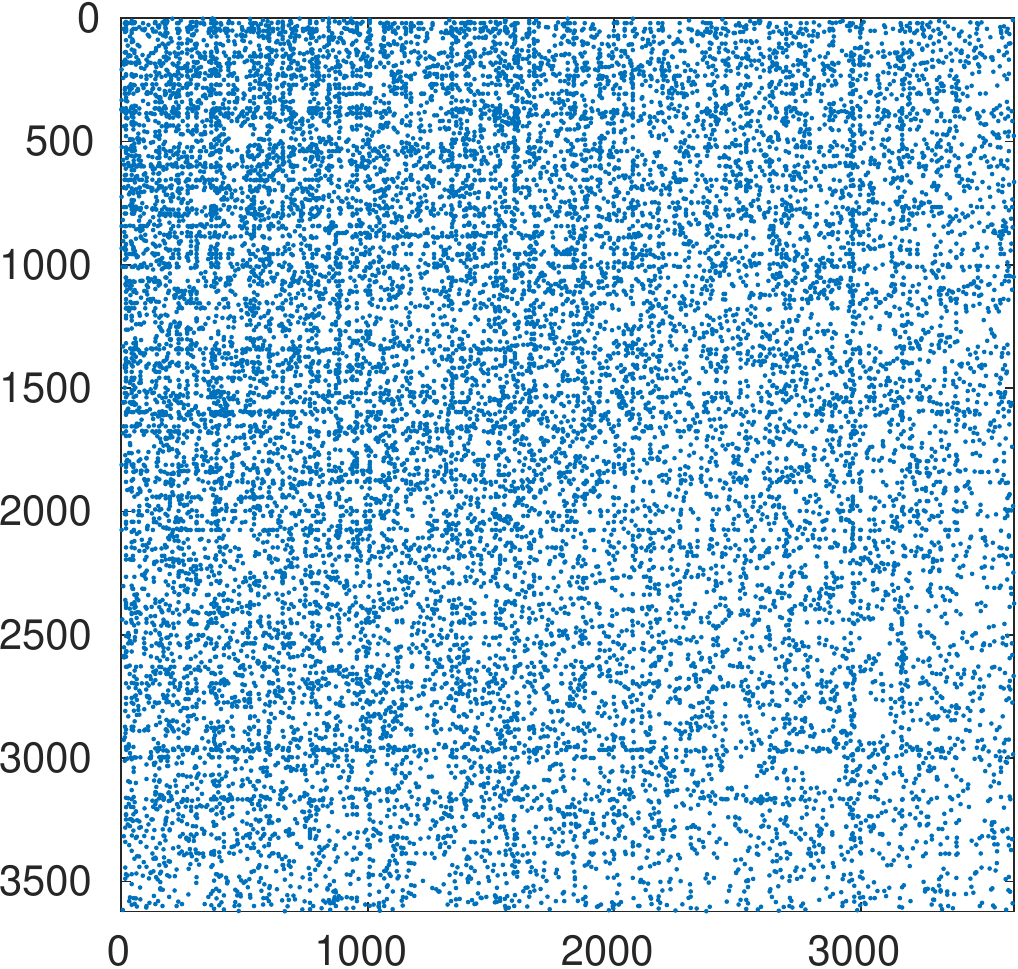}\hfill
\includegraphics[width=.49\textwidth]{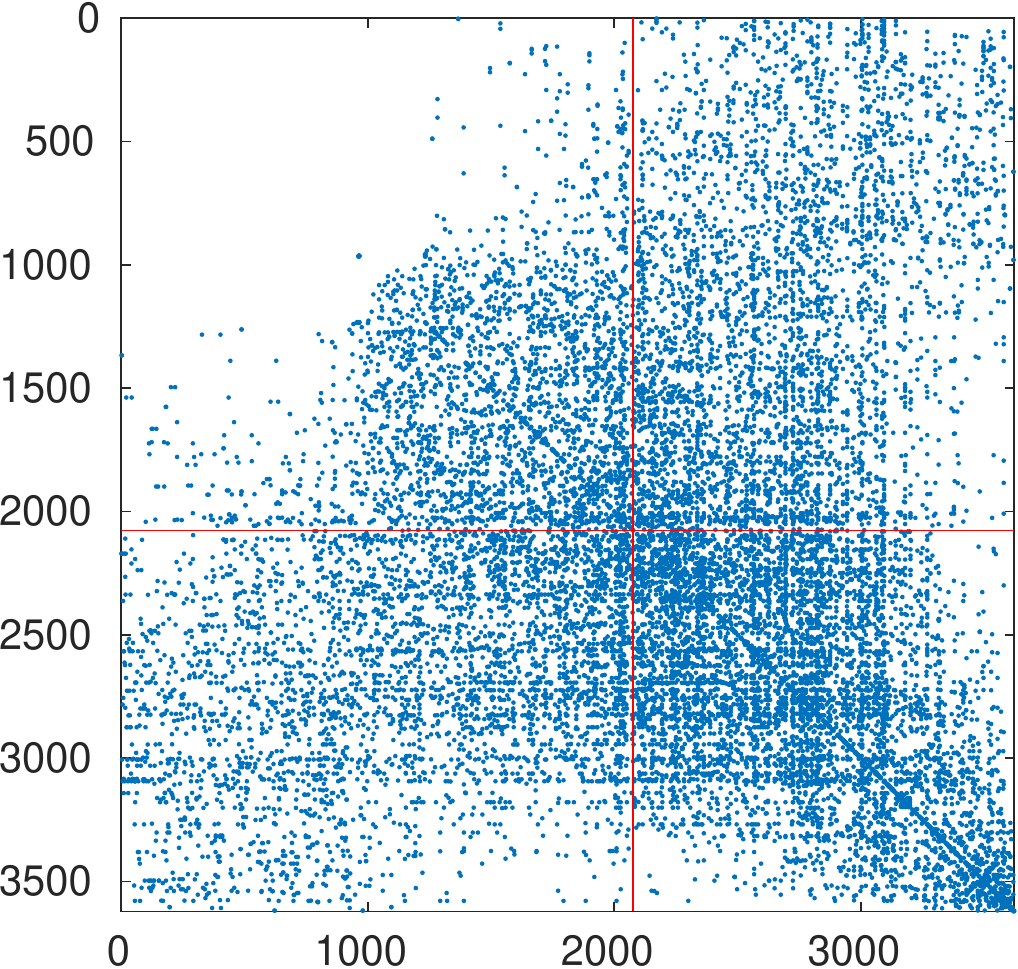}
\caption{Analysis of the \emph{geom} network: on the left, the starting
adjacency matrix; on the right, the node reordering produced by the spectral
algorithm.}
\label{ge1fig}
\end{center}
\end{figure}

\section{Conclusion}\label{sec6}

%Bipartization of a graph is a difficult problem, when it can be done.
This paper
describes how an approximate bipartization of a given graph can be determined
by solving a sequence of simple optimization problems. The technique can also 
be applied to detect anti-communities.
Computed examples illustrate the performance of the method described.

\section*{Acknowledgment}
The authors would like to thank the referees for comments.


\begin{thebibliography}{99}
\bibitem{bapat}
R. B. Bapat, Graphs and Matrices, Springer, London, 2010.
\bibitem{webpajek}
V. Batagelj and A. Mrvar, Pajek data sets (2006).
Available at \hfill\break\url{http://vlado.fmf.uni-lj.si/pub/networks/data/}
%\bibitem{RB86} 
%R. Bhatia, The distance between the eigenvalues of Hermitian matrices, Proc. Amer. Math.
%Soc., 96 (1986), pp. 41--42.
\bibitem{BM}
J. A. Bondy and U. S. R. Murty, Graph Theory with Applications, Macmillan, 
London, 1976.
\bibitem{BE}
S. P. Borgatti and M. G. Everett, Models of Core/Periphery Structures, Social
Networks, 21 (1999), pp. 375--395.
\bibitem{CYC}
L. Chen, Q. Yu, and B. Chen, Anti-modularity and anti-community detecting in
complex networks, Inf. Sci., 275 (2014), pp. 293--313.
\bibitem{CFR}
A. Concas, C. Fenu, and G. Rodriguez, PQser: A Matlab package for spectral seriation,
Numer. Algorithms, 80 (2019), pp. 879--902.
%\bibitem{DGK}
%P. Diaconis, R. L. Graham, and W. M. Kantor, The mathematics of perfect shuffles, Advances
%in Applied Mathematics, 4 (1983), pp. 175--196.
\bibitem{Es}
E. Estrada, The Structure of Complex Networks: Theory and Applications, Oxford University
Press, Oxford, 2011.
\bibitem{EGG}
E. Estrada and J. G\'omez--Garde\~nes, Network bipartivity and the transportation 
efficiency of European passenger airlines, Physica D, 323-324 (2016), pp. 57--63.
\bibitem{Esbook}
E. Estrada and P. A. Knight, A First Course in Network Theory, Oxford University
Press, Oxford, 2015.
\bibitem{FasTud}
D. Fasino and F. Tudisco, A modularity based spectral method for simultaneous community 
and anti-community detection, Linear Algebra Appl., 542 (2017), pp. 605--623.
\bibitem{MatBGL} 
D. Gleich, MatlabBGL - A Matlab Graph Library, \hfill\break
\url{https://www.cs.purdue.edu/homes/dgleich/packages/matlab_bgl/}
\bibitem{GVL}
G. H. Golub and C. F. Van Loan, Matrix Computations, 4th ed., Johns Hopkins University
Press, Baltimore, 2013.
\bibitem{H89} 
N. Higham, Matrix nearness problems and applications, in M. J. C. Gover and S. Barnett, 
eds., Applications of Matrix Theory, Oxford University Press, Oxford, 1989, pp. 1--27.
\bibitem{jmbo01}
H. Jeong, S. Mason, A.-L. Barab\'asi, and Z. N. Oltvai, Lethality and
centrality in protein networks, Nature, 411 (2001), pp. 41--42.
\bibitem{LLDM}
J. Leskovec, K. J. Lang, A. Dasgupta, and M. W. Mahoney, Statistical properties
of community structure in large social and information networks, in: Proc. 17th
Int. Conf. World Wide Web (WWW'08), 2008, pp. 695--704.
\bibitem{MBHG}
J. L. Morrison, R. Breitling, D. J. Higham, and D. R. Gilbert, A lock-and-key model for
protein-protein interactions, Bioinformatics, 2 (2006), pp. 2012--2019.
\bibitem{Ne}
M. E. J. Newman, Networks: An Introduction, Oxford University Press, Oxford, 2010.
%\bibitem{NPR} 
%S. Noschese, L. Pasquini, and L. Reichel, Tridiagonal Toeplitz matrices: properties and 
%novel applications, Numer. Linear Algebra Appl., 20 (2013), pp. 302--326.
\bibitem{PDFV}
G. Palla, I. Der\`enyi, I. Farkas, and T. Vicsek, Uncovering the overlapping
community structure of complex networks in nature and society, Nature, 435
(2005), pp. 814--818.
\bibitem{RAK}
U. N. Raghavan, R. Albert, and S. Kumara, Near linear time algorithm to detect
community structures in large-scale networks, Phys. Rev. E, 76 (2007), 036106.
\bibitem{RPFM}
M. P. Rombach, M. A. Porter, J. H. Fowler, and P. J. Mucha, Core-periphery
structure in networks, SIAM Rev., 59 (2017), pp. 619--646.
\bibitem{R}
R. Roth, On the eigenvectors belonging to the minimum eigenvalue of an essentially 
nonnegative symmetric matrix with bipartite graph, Linear Algebra Appl., 118 (1989), pp. 1--10.
\bibitem{slzrc03}
S. Sun, L. Ling, N. Zhang, G. Li, and R. Chen, Topological structure
analysis of the protein-protein interaction network in budding yeast, Nucleic
Acids Research, 31 (2003), pp. 2443--2450.
\bibitem{TVH}
A. Taylor, J. K. Vass, and D. J. Higham, Discovering bipartite substructure in
directed networks, LMS J. Comput. Math., 14 (2011), pp. 72--86.
%\bibitem{TB}
%L. N. Trefethen and D. Bau III, Numerical Linear Algebra, SIAM, Philadelphia, 1997.
\bibitem{YCL}
B. Yang, W. K. Cheung, and J. M. Liu, Community mining from signed social
networks, IEEE Trans. Knowl. Data Eng., 19 (2007), pp. 1333--1348.
\bibitem{Yann}
M. Yannakakis, Edge-deletion problems, SIAM J. Comput., 10 (1981),
pp. 297--309.
\end{thebibliography}
\end{document}